\documentclass[11pt,reqno]{amsart}
\usepackage{amsmath,amssymb,amsthm,amsrefs}
\usepackage{xcolor}

\newtheorem{thm}{Theorem}[section]
\newtheorem{pro}[thm]{Proposition}
\newtheorem{lem}[thm]{Lemma}
\newtheorem{cor}[thm]{Corolary}
\theoremstyle{definition}
\newtheorem{defi}[thm]{Definition}
\newtheorem{rem}[thm]{Remark}

\newtheorem{exa}[thm]{Example}

\newcommand{\nd}{\noindent}
\newcommand{\vu}{\vspace{.1cm}}
\newcommand{\vd}{\vspace{.2cm}}

\newcommand{\e}{\mathfrak{e}}

\begin{document}
\allowdisplaybreaks

\title[Partial (Co)Actions of Multiplier Hopf Algebras: Morita and Galois]{Partial (Co)Actions of Multiplier Hopf Algebras: Morita and Galois theories}

\author[Azevedo]{Danielle Azevedo}
\address[Azevedo]{Universidade Federal do Rio Grande do Sul, Brazil}
\email{danielle.azevedo@viamao.ifrs.edu.br}

\author[Batista]{Eliezer Batista}
\address[Batista]{Universidade Federal de Santa Catarina, Brazil}
\email{eliezer1968@gmail.com}

\author[Fonseca]{Graziela Fonseca}
\address[Fonseca]{Universidade Federal do Rio Grande do Sul, Brazil}
\email{grazielalangone@gmail.com}

\author[Fontes]{Eneilson Fontes}
\address[Fontes]{Universidade Federal do Rio Grande, Brazil}
\email{eneilsonfontes@furg.br}

\author[Martini]{Grasiela Martini}
\address[Martini]{Universidade Federal do Rio Grande, Brazil}
\email{grasiela.martini@furg.br}

\begin{abstract}
	In this work we deal with partial (co)actions of multiplier Hopf algebras on not necessarily unital algebras. Our main goal is to construct a Morita context relating the coinvariant algebra $R^{\underline{coA}}$ with a certain subalgebra of the smash product $R\#\widehat{A}$. Besides that, we present the notion of partial Galois coaction, which is closely related to this Morita context.
\end{abstract}

\maketitle

{\nd\scriptsize{\bf Key words:} multiplier Hopf algebra, partial action, smash product, Morita theory, Galois theory}\\
{\nd\scriptsize{\bf Mathematics Subject Classification:} primary 16T99; secondary 20L05}

\footnote{The first and the fifth authors were supported by CAPES, Brazil. The third author was supported by CNPq, Brazil.}

\section{introduction}
Partial group actions were introduced by R. Exel in \cite{exel}, within the context of operator algebras in order to describe some $\mathbb{Z}$-graded $C^*$-algebras which were not isomorphic to a crossed product. Some years later, M. Dockuchaev and R. Exel in \cite{exel2}, restated some results of partial group actions for a purely algebraic context, giving rise to several generalizations of classical results of the theory of group representations to the partial setting \cite{Mishasurvey}.

The subject of partial actions entered into the realm of Hopf algebras motivated by the theory of Galois extensions of commutative algebras by partial group actions \cite{DFP}. Extending the results in \cite{CdG}, in which the Galois theory for partial group actions could be described in terms of a Galois coring, S. Caenepeel and K. Janssen introduced the so called partial entwining structures, giving rise to the notion of a partial co-action of a Hopf algebra on a unital algebra and its dual counterpart, namely, partial actions of Hopf algebras on unital algebras 
\cite{Caenepeel}. In the sequel, many authors have explored these new structures providing a great development of this theory \cite{Batistasurvey}. One of the first results for partial actions of Hopf algebras was the globalization theorem, which states that every symmetric partial action of a Hopf algebra on a unital algebra can be viewed as a restriction of a global action of the same Hopf algebra on a bigger algebra 
\cite{MunizandBatista}. Another important result, whose generalization will be explored in the present article, is the Morita context between the subalgebra of partial invariants and the partial smash product and its relation with partial Galois extensions \cite{Marcelo}.

The core of the theory of partial Hopf actions is better understood under the conceptual framework of partial representations \cite{representation}. Given a Hopf algebra $H$, it is possible to introduce a Hopf algebroid $H_{par}$ such that partial modules of the original Hopf algebra correspond to modules over this new Hopf algebroid. The category of partial $H$-modules has a structure of a closed monoidal category and the partial actions of $H$ correspond to the algebra objects in this category. In the literature, partial actions of Hopf algebras were considered only on unital algebras. The case of partial actions of Hopf algebras over nonunital algebras was recently studied in \cite{Ferrazza}. There are some differences in this case concerning to globalization, instead of a general globalization theorem obtained in \cite{MunizandBatista}, in the case without unit, every partial action is Morita equivalent to a globalizable one, analogous to the result for partial actions of groups obtained in \cite{ADES}.

The duality between partial actions and coactions was explored first in \cite{MunizandBatista}, and more recently in \cite{Batista}. For finite dimensional Hopf algebras there exists an exact duality between partial right $H$-comodules and partial left $H^*$-modules. Going beyond the finite dimensional case introduces a conceptual difficulty concerning this duality because, in general, the finite dual, $H^{\circ}$ of a Hopf algebra $H$ is not ``big enough'' to present an exact duality between partial (co)actions. One way to circumvent this natural obstacle is to extend the theory to a class of objects which are well behaved with respect to duality and for which notions of action and coaction still make sense. This motivates one to work with multiplier Hopf algebras \cite{Multiplier}. 

The concept of a multiplier Hopf algebra was introduced by Alfons Van Daele, the original motivation was to introduce an algebraic framework for dealing with group duality \cite{Frame}. Given an infinite discrete group $G$, one can think the group algebra $\Bbbk G$ as the algebra of finitely supported functions defined on  the group $G$ taking values in $\Bbbk$ with the convolution product. This group algebra is a Hopf algebra, but its dual vector space, $A_G$, which is isomorphic to the non-unital algebra of functions of finite support from $G$ to $\Bbbk$, but now with the pointwise product. This dual algebra, cannot be endowed with a Hopf algebra structure. That is the point where multiplier Hopf algebras take place.

Multiplier Hopf algebras are, in general non-unital, but an important result, that appeared in \cite{Sweedler}, is the existence of bilateral local units for them. These local unities allow one to work rigorously with the Sweedler notation for the comultiplication in almost the same way as it is done for Hopf algebras. Another important consequence of the existence of local units is that the algebra is nondegenerate idempotent and this implies that the comultiplication is a nondegenerate algebra homomorphism (cf.  \cite[Appendix]{Multiplier}). Regular multiplier Hopf algebras, that is, multiplier Hopf algebras with bijective antipodes, are quite well behaved, having properties similar to finite dimensional Hopf algebras 
\cite{S4}, and the existence of invariant integrals on allows one to obtain an exact duality, that is, the dual vector space of a regular multiplier Hopf algebra is again a regular multiplier Hopf algebra \cite{Duality}. Therefore, multiplier Hopf algebras are useful to generalize results relative to Hopf algebras for the non-unital case in such a way that duality is preserved.

Actions of multiplier Hopf algebras were introduced in \cite{Action}, where also the smash product algebra was introduced and several results similar to those for Hopf algebras. Coactions of multiplier Hopf algebras, in their turn, were introduced in \cite{Galois}. In the case where the existence of integrals is ensured, one can prove a duality theorem between actions and coactions. Finally, a Morita context linking the smash product algebra and the coinvariant algebra was constructed.

Our aim in this article is to introduce partial actions and coactions of multiplier Hopf algebras over not necessarily unital algebras, extending classical results found in \cite{Action} and \cite{Galois}. In particular, we construct a Morita context between the partial smash product and the algebra of partial invariants and discuss its relationship with partial Galois extensions.

This paper is organized as follows: Section 2 is devoted to mathematical preliminaries. There the main results concerning regular multiplier Hopf algebras, used throughout the text, are introduced. In Section 3, we introduce the concept of a partial coaction of a regular multiplier Hopf algebra $A$ on a nondegenerated algebra $R$ extending both the classical theory of coactions of multiplier Hopf algebras \cite{Galois} and the theory of partial coactions of Hopf algebras \cite{Caenepeel}. In Section 4 the notion of a partial action of a regular multiplier Hopf algebra $A$ on a nondegenerate algebra $R$ is introduced. Furthermore, we discuss the duality between partial actions and partial coactions. Finally, section 5 is devoted to construct a Morita context, under appropriate conditions, connecting  a subalgebra of the smash product algebra $R\# \widehat{A}$, in which $\widehat{A}$ is the dual of the multiplier Hopf algebra $A$ defined by the left invariant integral, and the coinvariant algebra $R^{\underline{coA}}$. This construction generalizes those presented in \cite{Galois}, for the case of global actions of multiplier Hopf algebras, and in \cite{Marcelo}, for the case of partial actions of Hopf algebras. The connection between this Morita context and the notion of a partial Galois extension is discussed, following the same steps found in \cite{Galois}.

\section{Mathematical preliminaries}

Throughout this paper, vector spaces and algebras will be all considered over a fixed field $\Bbbk$. The symbol $\otimes$ will always mean $\otimes_{\Bbbk}$. Recalling, the algebra of the multipliers of any algebra $A$ over a field $\Bbbk$, denoted by $M(A)$, is the usual $\Bbbk$-vector space of all the ordered pairs $(U,V)$ of linear maps of $A$ that satisfy the following conditions: 
\begin{itemize}
	\item[(i)] $U(ab)=U(a)b$,
	\item [(ii)] $V(ab)=aV(b)$,
	\item [(iii)] $V(a)b=aU(b)$. 
\end{itemize} 
for all $a,b\in A$, endowed with the multiplication given by the rule $(U,V)(U',V')=(U\circ U',V'\circ V).$ Such an algebra is associative with identity element given by the pair $1=(\imath,\imath)$ where $\imath$ denotes the identity map of $A$. Furthermore, there exists a canonical algebra homomorphism $\jmath:A\to M(A)$ given by $a\mapsto (U_a,V_a)$, where $U_a$ (resp., $V_a$) denotes the left (resp., right) multiplication by $a$, for all $a\in A$. If, in particular, $\jmath$ is injective we say that the product in $A$ is \emph{nondegenerate}. In this case, $A$ is unital if and only if $\jmath$ is an isomorphism. Furthermore, every element $x$ of $M(A)$ will be denoted by the pair $(\overline{x}, \overline{\overline{x}})$ and the expression $\overline{x}(a)=xa$ (resp., $\overline{\overline{x}}(a)=ax$) will be seen as the product $xa$ (resp., $ax$) in $M(A)$, for all $a \in A$. As the product in $A$ is nondegenerate, the linear map $\overline{x}$ is univocally determinate by the linear map $\overline{\overline{x}}$ and conversely. Moreover,  a pair $(\overline{x}, \overline{\overline{x}})$ of linear maps from $A$ into $A$ lies in $M(A)$ if and only if the following compatibility relation $a\overline{x}(b)=\overline{\overline{x}}(a)b$ holds, for all $a,b \in A$.

Let $A$ be an algebra over a field $\Bbbk$ with a nondegenerate product. A \emph{comultiplication (or coproduct)} is an algebra homomorphism $\Delta: A\longrightarrow M(A\otimes A)$
satisfying
\begin{center}
	$\Delta(a)(1\otimes b)\in A\otimes A$\ \ \ \ \ and \ \ \ \ \  $(a\otimes 1)\Delta(b)\in A\otimes A$
\end{center}
and the co-associativity property 
\begin{center}
	$(a\otimes 1\otimes 1)((\Delta\otimes \imath)(\Delta(b)(1\otimes c)))=((\imath\otimes\Delta)((a\otimes 1)\Delta(b)))(1\otimes 1\otimes c),$
\end{center}
for all $a, b$,$c$ in $A$. A pair $(A,\Delta)$ is called a \textit{multiplier Hopf algebra} if $\Delta$ is a comultiplication and the linear maps 
\begin{eqnarray*}
	T_1: A\otimes A &\longrightarrow &A\otimes A \quad \ \ \ \ \ \ \ \ \ \mbox{and} \ \ \ \ \ \ \ \ \ \quad  T_2:A\otimes A \longrightarrow A\otimes A\\
	\quad \ \ \ \ \ \ a\otimes b &\longmapsto &\Delta(a)(1\otimes b) \ \ \ \ \ \ \ \  \ \ \ \ \ \ \ \ \ \ \ \ \ \ \quad a\otimes b \longmapsto (a\otimes 1)\Delta(b)
\end{eqnarray*}
are bijective.

Due to the surjectivity of these two maps it is possible to show that there exist a unique algebra homomorphism $\varepsilon :A\longrightarrow\Bbbk$ such that

\begin{center}
	$(\varepsilon\otimes \imath)(\Delta(a)(1\otimes b))=ab$ \ \ \ \ \ and \ \ \ \ \ $(\imath\otimes\varepsilon)((a\otimes 1)\Delta(b))=ab$
\end{center}
and a unique algebra anti-homomorphism $S: A\longrightarrow M(A)$ such that
\begin{center}
	$m(S\otimes \imath)(\Delta(a)(1\otimes b))=\varepsilon(a)b$ \ \ \ \ \ and \ \ \ \ \ $m(\imath\otimes S)((a\otimes 1)\Delta(b))=\varepsilon(b)a,$
\end{center}
for all $a,b$ in $A$.
Such maps are respectively called the \emph{counit} and the \emph{antipode} of $A$.
In particular, if $A$ is unital we recover the classical definition of a Hopf algebra.

Besides that, a multiplier Hopf algebra $(A,\Delta)$ is called \emph{regular} if $(A,\sigma\Delta)$  is also a multiplier Hopf algebra, where $\sigma$ denotes the canonical flip map. For regular multiplier Hopf algebras the antipode satisfies $S(A)= A$.

The motivation for the  concept of multiplier Hopf algebra arose from the algebra $A_G$, with pointwise product, of the complex functions with finite support on a group $G$, i. e., functions that assume nonzero values for a finite set of elements of $G$. In this case, the multiplier algebra $M(A_G)$ consists of all complex functions on $G$. Moreover, $A_G \otimes A_G$ can be naturally identified with the complex functions with finite support on $G\times G$. Then $A_G$ is a multiplier Hopf algebra with comultiplication $\Delta: A_G \longrightarrow M(A_{G\times G})$ given by $\Delta(f)(p,q)=f(pq)$, counit $\varepsilon:A_G\to \mathbb{C}$ given by $\varepsilon(f)=f(1_G)$ and antipode $S:A_G\to M(A_G)$ given by $(S(f))(p)=f(p^{-1})$, for all $f\in A_G$ and $p,q\in G$.

An important result, that appeared in \cite{Sweedler}, is the existence of bilateral local units for a multiplier Hopf algebra $(A, \Delta)$, that is, for any given finite set of elements $a_1, \ldots , a_n$ of $A$, there exists an element $e\in A$ such that $e a_i = a_i = a_i e$, for all $1 \leq i \leq n$. Such a fact was used to justify the Sweedler's notation in this context. Indeed, $\Delta(a)(1 \otimes b)$ can be written as $a_{(1)} \otimes a_{(2)}b$ since there is a local unit $e$ for the element $ b $, thus, $\Delta(a)(1 \otimes b)= \Delta(a)(1 \otimes e)(1 \otimes b)=a_{(1)} \otimes a_{(2)}b$. In this case, we say that the second entries $a_{(2)}$ is covered by $b.$ Another important consequence is that $A^2=A$, which allowed to show that the comultiplication $\Delta$ is a nondegenerate algebra homomorphism (cf.  \cite[Appendix]{Multiplier}).
 
One of the main differences between Hopf algebras and multiplier Hopf algebras is related to the linear duality. In Hopf algebra case, one deals with the finite dual Hopf algebra. By the way, the finite dual can be very small, for example, in the case of simple Hopf algebras. For multiplier Hopf algebras, the construction of the dual depends on the existence of integrals \cite{Galois}. 

 Given a regular multiplier Hopf algebra $(A, \Delta)$, a linear functional $\varphi$ on $A$ is called a left integral if $(\imath\otimes\varphi)\Delta(a)=\varphi(a)$, for all $a\in A$. Define the dual algebra 
  \begin{eqnarray}\label{achapeu}
 \widehat{A}=\{\varphi(\underline{\hspace{0.3cm}} a); \  a\in A, \varphi\ \text{is a left  integral} \} ,
 \end{eqnarray}
 whose product and coproduct are defined as follows: For $u, w\in \widehat{A}$, their product is given by
 \begin{center}
	$(wu)(a)=(w\otimes u)\Delta(a)$,	
\end{center} 
and the coproduct, for any $w\in \widehat{A}$
\begin{eqnarray}
\label{deltadual1} (\widehat{\Delta}(w)(1\otimes u))(a\otimes b) &=& (w\otimes u)((a\otimes 1)\Delta(b)),\\
\label{deltadual2} ((u\otimes 1)\widehat{\Delta}(w))(a\otimes b) &=& (u\otimes w)(\Delta(a)(1\otimes b)),
\end{eqnarray} 
for all $a,b\in A$. $(\widehat{A}, \widehat{\Delta})$ is indeed an example of a multiplier Hopf algebra with integrals and $ \widehat{\widehat{A} \ }\cong A$.  Throughout the text, unless otherwise stated, we will use $\widehat{a}=\varphi(\underline{\hspace{0.3cm}} a)$ to not overload the notation.

The following identites involving the elements of the dual multiplier Hopf algebra $\widehat{A}$ will be useful throughout the text:
\begin{eqnarray}
\label{escritadelta1} \widehat{\Delta} (\varphi (\underline{\hspace{0.3cm}} a)) (1\otimes \varphi (\underline{\hspace{0.3cm}} b)) & =& \varphi (\underline{\hspace{0.3cm}} S^{-1}(b_{(1)})a) \otimes \varphi (\underline{\hspace{0.3cm}} b_{(2)}) ,\\
\label{escritadelta2} (\imath \otimes \widehat{S})((1\otimes \widehat{S^{-1}}(\varphi (\underline{\hspace{0.3cm}} b)))\widehat{\Delta} (\varphi (\underline{\hspace{0.3cm}} a))  ) & =& \varphi (\underline{\hspace{0.3cm}} b_{(1)}a) \otimes \varphi (\underline{\hspace{0.3cm}} b_{(2)}) .
\end{eqnarray}
  
For completeness sake, we give a sketch of the proof of identity (\ref{escritadelta1}), since these auxiliary identities were obtained by the authors and do not appear anywhere else in the literature on multiplier Hopf algebras. Indeed, given $\varphi (\underline{\hspace{0.3cm}} a), \varphi (\underline{\hspace{0.3cm}} b)\in\widehat{A}$ and $c,d\in A$,
\begin{eqnarray*}
	\widehat{\Delta}(\varphi (\underline{\hspace{0.3cm}} a))(1\otimes\varphi (\underline{\hspace{0.3cm}} b))(c\otimes d)&\stackrel{(\ref{deltadual1})}{=}& (\varphi (\underline{\hspace{0.3cm}} a)\otimes\varphi (\underline{\hspace{0.3cm}} b))((c\otimes 1)\Delta(d))\\
	&=&(\varphi\otimes\varphi)((c\otimes 1)\Delta(d)(a\otimes b))\\
	&=& (\varphi\otimes\varphi)((c\otimes 1)\Delta(d)\Delta(b_2)(S^{-1}(b_1)a\otimes 1))\\
	&=& \varphi(c(db_2)_1S^{-1}(b_1)a)\varphi((db_2)_2)\\
	&\stackrel{(\ast)}{=}& \varphi(cS^{-1}(b_1)a)\varphi(db_2)\\
	&=&(\varphi(\underline{\hspace{0.3cm}}\ S^{-1}(b_1)a)\otimes \varphi(\underline{\hspace{0.3cm}} \ b_2))(c\otimes d),
\end{eqnarray*}
where in $(\ast)$ we used the left invariance of the integral $\varphi$.
\vu

For example, $\widehat{\Bbbk G} \cong A_G$ given by
\begin{eqnarray*}
\theta: \widehat{\Bbbk G} &\longrightarrow & A_G\\
\varphi(\underline{\hspace{0.3cm}} a)&\longmapsto & \theta(\varphi(\underline{\hspace{0.3cm}} a))(h)=\varphi(\delta_ha),
\end{eqnarray*}
where $a=\displaystyle\sum_{g\in G}a_g\delta_g$, i. e., $\theta(\varphi(\underline{\hspace{0.3cm}} a))(h)=\displaystyle\sum_{g\in G}a_g\varphi(\delta_hg)=a_{h^{-1}}$. Moreover, $\widehat{A_G}\cong \Bbbk G$.

 \section{Partial coactions}

\subsection{Global coactions}

In this section we recall the definition of comodule algebra and some properties, as in \cite{Galois}.

\begin{defi} Let $A$ be a regular multiplier Hopf algebra and $R$ an algebra. We call $R$ a \textit{right $A$-comodule algebra} if there exists an injective homomorphism $\rho : R\longrightarrow M(R\otimes A)$ satisfying
		\begin{enumerate}
		\item[(i)] $\rho(R)(1\otimes A)\subseteq R\otimes A$ and $(1\otimes A)\rho(R)\subseteq R\otimes A$;
		\item[(ii)] $(\rho\otimes \imath)\rho=(\imath\otimes \Delta)\rho$.
	\end{enumerate}
	
	In this case, the map $\rho$ is called a coaction of $A$ on $R$. We say that a coaction $\rho$ is \textit{reduced}, if  $(R\otimes 1)\rho(R)\subseteq R\otimes A$ also holds.
	\label{defcoacao}
\end{defi}

\begin{rem} Using (i), the co-associativity in (ii) can be viewed as follows:
	\begin{eqnarray*}
		(\rho\otimes \imath)(\rho(x)(1\otimes b))=(\imath\otimes\Delta)(\rho(x))(1\otimes 1\otimes b),
	\end{eqnarray*}
	for all $x\in R$ and $b\in A$.
\end{rem}

\begin{pro} If $R$ is a right $A$-comodule algebra via $\rho$, then $(\imath\otimes\varepsilon)\rho(x)=x$, for all $x\in R$.
\end{pro}

\begin{pro} If $R$ is a right $A$-comodule algebra via $\rho$, then the linear maps
	\begin{eqnarray*}
		T_1: R\otimes A &\longrightarrow & R\otimes A\ \ \ \ \ \ \ \ \ \ \mbox{and}\ \ \ \ \ \ \ \ \ \ T_2: R\otimes A\longrightarrow R\otimes A\\
		x\otimes a &\longmapsto &\rho(x)(1\otimes a) \ \ \ \ \ \ \ \ \ \ \ \ \ \ \ \ \ \ \ \ \ x\otimes a\longmapsto (1\otimes a)\rho(x)
	\end{eqnarray*}
	are bijective.
	\label{procoacaobij}
\end{pro}
These above bijections imply $\rho(R)(1\otimes A)=R\otimes A=(1\otimes A)\rho(R)$. Hence,
\begin{center}
	$\rho(R)(R\otimes A)=R^2\otimes A=(R\otimes A)\rho(R)$,
\end{center}
what means that $\rho$ is a nondegenerate homomorphism if $R^2=R$.

Although the coaction $\rho$ is not a nondegenerate homomorphism, it can be uniquely extended to $ M (R) $, using the bijectivity of the linear maps $ T_1 $ and $ T_2 $ as follows:
\begin{eqnarray}
	\rho: M(R) &\longrightarrow & M(R\otimes A)\label{coaglobal}\\
	m&\longmapsto & \rho(m)=(\overline{\rho(m)},\overline{\overline{\rho(m)}})\nonumber
\end{eqnarray}
such that $\overline{\rho(m)}(\rho(x)(1\otimes a))=\rho(mx)(1\otimes a)$ and $\overline{\overline{\rho(m)}}((1\otimes a)\rho(x))=(1\otimes a)\rho(xm)$, for all $x\in R$ and $a\in A$. 

\begin{lem} If $R$ is a reduced right $A$-comodule algebra, then  $\rho(R)(R\otimes 1)\subseteq R\otimes A$.
\end{lem}

\subsection{Partial coactions}

From now on, the partial coactions of $A$ on $R$ will be always consider on the right. The partial coactions on the left are defined in a similar way. Firstly, we recall the definition of a partial coaction  when $A$ and $R$ are unital. 

\begin{defi} \label{right_part_comod_alg} \cite{Marcelo} Let $A$ be a Hopf algebra. An algebra $R$ is a \textit{partial $A$-comodule algebra} if there exists a linear map
	\begin{eqnarray*}
		\rho: R &\longrightarrow & R \otimes A\\
		x & \longmapsto & x^{(0)} \otimes x^{(1)}
	\end{eqnarray*}
such that 
	\begin{enumerate}
		\item[(i)] $\rho(xy)= \rho(x)\rho(y)$;
		\vu
		\item [(ii)] $(\imath\otimes \varepsilon)\rho(x)=x$;
		\vu
		\item [(iii)] $(\rho \otimes \imath)\rho(x)=(\rho(1_R)\otimes 1_A)(\imath\otimes \Delta)\rho(x)$,
	\end{enumerate} for all $x,y\in R$. In this case, $\rho$ is called a \emph{partial coaction} of $A$ on $R$.
\end{defi} 

The partial coaction $\rho$ is called \textit{symmetric} if, in addition, satisfies:
 \begin{enumerate}
 	\item[(iv)]$(\rho\otimes\imath)\rho(x)= (\imath\otimes \Delta)\rho(x)(\rho(1_R)\otimes 1_A)$, for all $x\in R$.
\end{enumerate}

\vu
For the context where both algebras are nonunital, clearly the items (iii) and (iv) do not make sense anymore. In reference \cite{Batista}, the authors proved that the image of the partial coaction lies in a direct summand of the tensor product $R\otimes A$ and the projection is given by $\rho (1_R)$. The first idea to overcome this problem in the nonunital case is to look for a new projection  playing the same role of $\rho (1_R)$. Inspired by \cite{Fraca0}, in a similar problem for the context of weak multiplier Hopf algebras, where there is an idempotent $E\in M(A\otimes A)$ which coincides with $\Delta (1_{M(A)})$ when the comultiplication is extended to $M(A)$, we define an idempotent $E\in M(R\otimes A)$ which will coincide with $\rho (1_{M(R)})$ when the coaction is extended to $M(R)$.

\begin{defi} Let $A$ be a regular multiplier Hopf algebra and $R$ an algebra with a nondegenerate product. We call $(R, \rho, E)$ (or simply $R$) a \emph{partial $A$-comodule algebra} if $\rho: R \longrightarrow M(R\otimes A)$ is an injective algebra homomorphism  and $E\in M(R\otimes A)$ is an idempotent element, satisfying  
 	\begin{enumerate}
 		\item[(i)] $(1\otimes A)E\subseteq M(R)\otimes A$ and $E(1\otimes A)\subseteq M(R)\otimes A$;
 		\vu
 		\item[(ii)] $\rho(R)(1\otimes A)\subseteq E(R\otimes A)$ and $(1\otimes A)\rho(R)\subseteq(R\otimes A)E$; 
 		\vu
 		\item[(iii)] $(\rho\otimes \imath)(\rho(x)) = (E\otimes 1)(\imath\otimes \Delta)(\rho(x))$,
 	 	\end{enumerate} 
for all $x\in R$. In this case, $\rho$ is called a \emph{partial coaction} of $A$ on $R$. We say that the partial coaction $\rho$ is \emph{symmetric} if, besides the above conditions,  $\rho$ also satisfies
 \begin{enumerate}
 	\item[(iv)] $(\rho \otimes \imath)(\rho(x))=(\imath \otimes \Delta)(\rho(x))(E\otimes 1)$,  for all $x\in R$.
 \end{enumerate}
 \label{def_comoalgparcarcial_multip} 	
 	   \end{defi}

 Similarly to the global case, we use the condition (ii) to rewrite the other ones as follows: 
 \begin{eqnarray}
(\rho \otimes \imath)(\rho(x)(1\otimes b))&=&(E\otimes 1)(\imath\otimes \Delta)(\rho(x))(1\otimes 1\otimes b)\label{escritacoa1},\\
 (\rho \otimes \imath)((1\otimes b)\rho(x))&=&(1\otimes 1\otimes b)(E\otimes 1)(\imath\otimes \Delta)(\rho(x))\label{escritacoa2},\\
(\rho \otimes \imath)(\rho(x)(1\otimes b)) &=& (\imath\otimes \Delta)(\rho(x))(E\otimes 1)(1\otimes 1 \otimes b)\label{escritacoa3}\\
 (\rho \otimes \imath)((1\otimes b)\rho(x)) &=& (1\otimes 1 \otimes b)(\imath\otimes \Delta)(\rho(x))(E\otimes 1)\label{escritacoa4}.
 \end{eqnarray}

\begin{rem} Every $A$-comodule algebra $R$ is a symmetric partial comodule algebra, taking the  idempotent $E=1_{M(R)}\otimes 1_{M(A)}$.
\end{rem}

\begin{lem} Let $(R,\rho,E)$ be a partial comodule algebra. Then
\begin{equation}
E\rho(x)=\rho(x) \ \ \mbox{and} \ \ \rho(x)E=\rho(x),
\end{equation}
for all $x\in R$.
\label{rhoigual_rhoe}
\label{lema_rhoigroe}
\end{lem}
\begin{proof}
By assumption $\rho(R)(R\otimes A)=\rho(R)(1\otimes A)(R\otimes 1)\subseteq E(R\otimes A)$, thus
\begin{eqnarray*}
\rho(x)(y\otimes a)&=&E(\sum_i z_i\otimes b_i)\\
&= & EE(\sum_i z_i\otimes b_i)\\
&=&E\rho(x)(y\otimes a),
\end{eqnarray*}
for all $x,y\in R$ and $a\in A$. Therefore, $\rho(x)=E\rho(x)$, for all $x\in R$. Similarly, $\rho(x)=\rho(x)E$, for all $x\in R$.
\end{proof}

   \begin{pro} Let $(R,\rho,E)$ be a partial $A$-comodule algebra. Then $R$ is an $A$-comodule algebra via $\rho$ if and only if $E=1_{M(R)}\otimes 1_{M(A)}$.
\label{caract_comod}
\end{pro}
\begin{proof}
Assume that $R$ is an $A$-comodule algebra. Thus,
\begin{center}
$R\otimes A=\rho(R)(1\otimes A)\subseteq E(R\otimes A) .$\\

\end{center}

Therefore, for all $x\in R$ and $a\in A$,
$$x\otimes a= E(\sum_i y_i\otimes b_i)= E E(\sum_i y_i\otimes b_i)= E(x\otimes a).$$

Then  $E=1_{M(R)}\otimes 1_{M(A)}$. Conversely, if $E=1_{M(R)}\otimes 1_{M(A)}$ we naturally obtain Definition \ref{defcoacao}.
\end{proof}

 \begin{pro}
	If $(R, \rho, E)$ is a partial $A$-comodule algebra, then  $(i\otimes \varepsilon)\rho(x)=x$, for all $x\in R$.
	\label{pro_coavarepsi}
\end{pro}
\begin{proof}
Let $b\in A$ such that $\varepsilon(b)=1_{\Bbbk}$, hence
\begin{eqnarray*}
\rho((\imath\otimes\varepsilon)(\rho(x)(1\otimes b)))&=& (\imath\otimes\imath\otimes\varepsilon)((\rho\otimes\imath)(\rho(x)(1\otimes b)))\\
&\stackrel{(\ref{escritacoa1})}{=}&(\imath\otimes\imath\otimes\varepsilon)((E\otimes 1)(\imath\otimes\Delta)(\rho(x))(1\otimes 1\otimes b))\\
&=& E(\imath\otimes (\imath\otimes\varepsilon)\Delta)(\rho(x))\varepsilon(b)\\
&=&E\rho(x)\\
&\stackrel{\ref{rhoigual_rhoe}}{=}&\rho(x),
\end{eqnarray*}
for all $x\in R$. Therefore, since $\rho$ is an injective map, $x=(\imath\otimes\varepsilon)(\rho(x)(1\otimes b))=(\imath\otimes\varepsilon)\rho(x)$, for all $x\in R$.
\end{proof}

Lemma \ref{lema_rhoigroe} is also used in the following result.

\begin{pro} Let $(R, \rho, E)$ be a partial $A$-comodule algebra. Then
\begin{eqnarray*}
(\rho\otimes \imath)((y\otimes 1)\rho(x)(1\otimes b))=(\rho(y)\otimes 1)(\imath\otimes\Delta)(\rho(x))(1\otimes 1\otimes b),
\end{eqnarray*}
for all $b\in A$ and $x,y\in R$.
\end{pro}
\begin{proof}
In fact,
\begin{eqnarray*}
(\rho\otimes \imath)((y\otimes 1)\rho(x)(1\otimes b))&=& (\rho(y)\otimes 1)(\rho\otimes \imath)(\rho(x)(1\otimes b))\\
&=&(\rho(y)\otimes 1)(E\otimes 1)(\imath\otimes\Delta)(\rho(x))(1\otimes 1\otimes b)\\
&=& (\rho(y)E\otimes 1)(\imath\otimes\Delta)(\rho(x))(1\otimes 1\otimes b)\\
&=& (\rho(y)\otimes 1)(\imath\otimes\Delta)(\rho(x))(1\otimes 1\otimes b),
\end{eqnarray*}
for all $b\in A$ and $x,y\in R$.
\end{proof}
Similarly, if the partial coaction $\rho$  is symmetric
\begin{eqnarray*}
	(\rho\otimes\imath)((1\otimes b)\rho(x)(y\otimes 1))=(1\otimes 1\otimes b)(\imath\otimes\Delta)(\rho(x))(\rho(y)\otimes 1),
\end{eqnarray*}
for all $b\in A$ and $x,y\in R$.

\begin{rem} The items of Definition \ref{def_comoalgparcarcial_multip} can be rewritten as follows:
\begin{itemize}
\item[(i)] $\rho(x)(1\otimes a)=x^{(0)}\otimes x^{(1)}a$ and $(1\otimes a)\rho(x)=x^{(0)}\otimes ax^{(1)}\in R\otimes A$;
\vu
\item[(ii)] $x^{(0)(0)}\otimes x^{(0)(1)}a\otimes x^{(1)}b=\sum_i E(x^{(0)}\otimes (x^{(1)}a_i)_{(1)})\otimes (x^{(1)}a_i)_{(2)}b_i $, denoting $a\otimes b=\sum_i \Delta(a_i)(1\otimes b_i)$;
\vu
\item[(iii)] $x^{(0)(0)}\otimes ax^{(0)(1)}\otimes bx^{(1)}=\sum_{i,j}m_ix^{(0)}\otimes (a_{ij}x^{(1)})_{(1)}\otimes b_j(a_{ij}x^{(1)})_{(2)}$, denoting $(1\otimes a)E=\sum_im_i\otimes a_i$ and, for each $i$, $a_i\otimes b=\sum_j(1\otimes b_j)\Delta(a_{ij})$;
\vu
\item[(iv)] $x^{(0)(0)}\otimes x^{(0)(1)}a\otimes x^{(1)}b=\sum_{i,j}x^{(0)}m_i\otimes (x^{(1)}a_{ij})_{(1})\otimes (x^{(1)}a_{ij})_{(2)}b_j$, denoting $E(1\otimes a)=\sum_i m_i\otimes a_i$ and, for each $i$, $a_i\otimes b=\sum_{j}\Delta(a_{ij})(1\otimes b_j)$;
\vu
\item[(v)] $x^{(0)(0)}\otimes ax^{(0)(1)}\otimes bx^{(1)}=\sum_i(x^{(0)}\otimes (a_ix^{(1)})_{(1)})E\otimes b_i(a_ix^{(1)})_{(2)}$, denoting $a\otimes b=\sum_i(1\otimes b_i)\Delta(a_i)$,
\end{itemize}
for all $x\in R$ and $a,b\in A$,
 \end{rem}
\begin{pro} Let $(R, \rho, E)$ be a partial $A$-comodule algebra. Then $$\rho(R)(1\otimes A)=E(R\otimes A).$$
\label{pro_bijetivcoacao}
\end{pro}
\begin{proof} \vspace{-0.7cm}
It is enough to check that $E(R\otimes A)\subseteq \rho(R)(1\otimes A)$. Indeed, for any $a\in A$ one can write $(1\otimes a)E=\sum_i m_i \otimes a_i$. Then, taking $d\in A$ such that $\varepsilon(d)=1_{\Bbbk}$, we have
\begin{eqnarray*}
&&\hspace{-2cm}(1\otimes a)E(x\otimes b)=\\
&\stackrel{\ref{pro_coavarepsi}}{=}& \sum_i m_ix^{(0)} \varepsilon (x^{(1)}d)\otimes a_iS(S^{-1}(b))\\
&=& \sum_i m_ix^{(0)}\varepsilon(S^{-1}(b)_{(2)}(x^{(1)}d))\otimes a_iS(S^{-1}(b)_{(1)})\\
&=&\sum_i m_ix^{(0)}\otimes a_iS(S^{-1}(b)_{(1)})\varepsilon(S^{-1}(b)_{(2)}x^{(1)})\varepsilon (d)\\
&=& \sum_i m_ix^{(0)}\otimes m(\imath\otimes S)((a_iS(S^{-1}(b)_{(1)})\otimes 1)\Delta(S^{-1}(b)_{(2)}x^{(1)}))\\
&=&  (\imath\otimes m(\imath\otimes S))((\sum_im_i\otimes a_iS(S^{-1}(b)_{(1)})\otimes 1)(\imath\otimes\Delta)((1\otimes S^{-1}(b)_{(2)})\rho(x)))\\
&=& (\imath\otimes m(\imath\otimes S))((\sum_im_i\otimes( a_iS(S^{-1}(b)_{(1)})\otimes 1)\Delta(S^{-1}(b)_{(2)}))(\imath\otimes\Delta)(\rho(x)))\\
&=&(\imath\otimes m(\imath\otimes S))((\sum_im_i\otimes a_i\otimes S^{-1}(b))(\imath\otimes\Delta)(\rho(x)))\\
&=&(\imath\otimes m(\imath\otimes S))(((1\otimes a)E\otimes S^{-1}(b))(\imath\otimes\Delta)(\rho(x))) \\
&=&(\imath\otimes m(\imath\otimes S))((1\otimes a\otimes S^{-1}(b))(E\otimes 1)(\imath\otimes\Delta)(\rho(x)))\\
&\stackrel{(\ref{escritacoa2})}{=}& (\imath\otimes m(\imath\otimes S))((1\otimes a\otimes 1)(\rho\otimes\imath)((1\otimes S^{-1}(b))\rho(x)))\\
&=&(\imath\otimes m(\imath\otimes S))(x^{(0)(0)}\otimes ax^{(0)(1)}\otimes S^{-1}(b)x^{(1)})\\
&=&x^{(0)(0)}\otimes ax^{(0)(1)}S(S^{-1}(b)x^{(1)})\\
& = & (1\otimes a)\rho(x^{(0)})(1\otimes S(S^{-1}(b)x^{(1)})),
\end{eqnarray*}
for all $x\in R$ and $a,b\in A$. Hence, $E(x\otimes b)=\rho(x^{(0)})(1\otimes S(S^{-1}(b)x^{(1)}))\in \rho(R)(1\otimes A)$.
\end{proof}

Similarly to the above result, if $R$ is a symmetric partial $A$-comodule algebra then $(1 \otimes A)\rho(R)=(R \otimes A)E$.

\begin{pro} If $A$ and $R$ are unital,  then Definition \ref{right_part_comod_alg} and Definition \ref{def_comoalgparcarcial_multip} coincide.
\end{pro}
\begin{proof}
Suppose Definition \ref{right_part_comod_alg}. It is enough to consider $E=\rho(1_R)$ and  to observe that the item (ii) of this definition is equivalent to the injectivity of the coaction $\rho$. Conversely,
$$
\rho(1_R)(x\otimes a)\stackrel{\ref{rhoigual_rhoe}}{=}\rho(1_R)E(x\otimes a)\stackrel{\ref{pro_bijetivcoacao}}{=} \sum_i\rho(1_Rx_i)(1\otimes a_i)= E(x\otimes a),$$
for all $x\in R$ and $a\in A$, hence $\rho(1_R)=E$.
\end{proof}

Our purpose now is to extend a symmetric partial coaction of $A$ on $R$ to an algebra homomorphism $\rho :M(R) \longrightarrow M(R \otimes A)$.
  \begin{pro}  	\label{pro_extcoapart}
  	Let $(R,\rho,E)$ be a symmetric partial  $A$-comodule algebra. Then there exists a unique algebra homomorphism $\rho: M(R)\longrightarrow M(R\otimes A)$ such that $\rho(1_{M(R)})=E$.
  \end{pro}
  \begin{proof}	By assumption $E(R\otimes A)=\rho(R)(1\otimes A)$ and $(R\otimes A)E=(1\otimes A)\rho(R)$, then it is enough to define the following linear map
 	\begin{eqnarray*}
 		\rho: M(R)&\longrightarrow &M(R\otimes A)\\
 		m &\longmapsto &\rho(m)=(\overline{\rho(m)},\overline{\overline{\rho(m)}}),
 	\end{eqnarray*}
 	such that $\overline{\rho(m)}(x\otimes a)=\sum_i\rho(my_i)(1\otimes b_i)$, where $E(x\otimes a)=\sum_i\rho(y_i)(1\otimes b_i)$, and $\overline{\overline{\rho(m)}}(x\otimes a)=\sum_j (1\otimes c_j)\rho(z_jm)$, where $(x\otimes a)E=\sum_j (1\otimes c_j)\rho(z_j)$, for all $x\in R$ and $a\in A$.
 \end{proof}

 \begin{cor} Under the above hypothesis, the linear map $\rho: M(R)\longrightarrow M(R\otimes A)$ is injective.
 \end{cor}

Our approach differs from \cite{timmermann}, where the existence of an extension for $\rho$ to the multiplier algebra is given as an hypothesis in the definition of the partial coaction. There, this extension follows directly from properties of continuous morphisms between $C^*$ algebras and the existence of approximate units. In the purely algebraic context it is not natural to assume a priori the existence of such an extension. Moreover, the equality in Proposition \ref{pro_bijetivcoacao} and the extension itself, in Proposition \ref{pro_extcoapart}, must be put by hand in their approach, while in our case these results are natural consequences of the definition.

\subsection{Examples of partial coactions}
\begin{pro}\label{coalambda} Let $\rho:R\rightarrow M(R\otimes A)$ be a linear map given by $\rho(x)=x\otimes m$, where $m\in M(A)$ and $m\neq 0$. Then $\rho$ is a partial coaction of $A$ to $R$ if and only if $m$ satisfies:
\begin{enumerate}
\item[(i)] $m^2=m$;
\item[(ii)] $m\otimes m=(m\otimes1)\Delta(m)$.
\end{enumerate}
Furthermore, the condition of symmetry of Definition \ref{def_comoalgparcarcial_multip} is equivalent to $m \otimes m= \Delta(m)(m \otimes 1)$.

	\end{pro}
\begin{proof}
	The proof is immediate by taking $E=1\otimes m$.
\end{proof}	

\begin{rem} If $m\in M(A)$ satisfies the condition  $m\otimes m
	=(m\otimes1)\Delta(m)$, then $m^2=m$ if and only if $\varepsilon(m)=1_{\Bbbk}.$ 
\end{rem}

\begin{exa} Consider the algebra $A_G$ as in Example \ref{chato} and $R$ any algebra with a nondegenerate product. The linear map $\rho: R\longrightarrow M(R\otimes A_G)$ given by $\rho(x)=x\otimes m\in R\otimes M(A_G) \subseteq M(R \otimes A_G)$  is a symmetric partial coaction if and only if
	\begin{eqnarray*}
		m: G & \longrightarrow & \Bbbk\\
		g & \longmapsto & \left\{
		\begin{array}{rl}
			1, & \text{ if } g\in N\\
			0, & \text{ otherwise, }
		\end{array} \right.
	\end{eqnarray*}
	where $N$ is any subgroup of $G$.
	\label{ex_mlambda}
	\end{exa}

\begin{exa}\label{exestrito}
	Under the same above condition, $R$ is a symmetric partial $A_G$-comodule algebra via
\begin{eqnarray*}
\rho: R &\longrightarrow &R\otimes A_G\\
x &\longmapsto & x\otimes\delta_{1_G},
\end{eqnarray*}	
where $1_G$ denotes the identity element of $G$.
\end{exa}

The next proposition extends the notion of induced partial coaction, presented in \cite{MunizandBatista}, for the context of regular multiplier Hopf algebras.

\begin{pro}\textnormal{(Induced Partial Coaction)} \label{pro_coainduzido} Let $R$ be an $A$-comodule algebra via $\rho$ and $L$ a right ideal of $R$ with identity $1_L$. Then
	\begin{eqnarray*}
		\beta: L &\longrightarrow & M(L\otimes A)\\
		l &\longmapsto &\beta(l):=(1_L\otimes 1_{M(A)})\rho(l)=(1_L\otimes 1)\rho(l)
	\end{eqnarray*}
	is a partial coaction of $A$ on $L$. In this case, $\beta$ is called an induced partial coaction.
\end{pro}
\begin{proof}
	It easily follows by taking $E=(1_L\otimes 1)\rho(1_L)\in M(L\otimes A)$.
\end{proof}

Observe that, if the algebra $L$ is a bilateral ideal of $R$, then its unit $1_L$ is a central idempotent element in $R$. In this case, the induced partial coaction is symmetric.

\begin{exa} Consider $A_G$ as the $A_G$-comodule algebra via its coproduct $\Delta$, where $A_G$ was defined in Example \ref{chato}. Take $N$ a finite subgroup of $G$, $f_N=\sum\limits_{n\in N}\delta_n\in A_G$ a central idempotent  and $L=f_NA_G$. By Proposition \ref{pro_coainduzido}, $L$ is a symmetric partial $A_G$-comodule algebra via
	\begin{eqnarray*}
		\beta: L&\longrightarrow & M(L\otimes A)\\
		f_N\delta_p &\longmapsto &(f_N\otimes 1)\Delta(f_N\delta_p),
	\end{eqnarray*}
	where $E=(f_N\otimes 1)\Delta(f_N)$. Note that $\beta$ is not global, since given $h\in N$ and $p\in G$, such that $p \notin N$,
	\begin{eqnarray*}
		E(h\otimes p)=(f_N\otimes 1)\Delta(f_N)(h\otimes p)=\sum_{m,n\in N} \delta_m(h)\delta_n(hp)=0,
	\end{eqnarray*}
	and, on the other hand, 
	$$(1_S\otimes 1)(h \otimes p)=(\sum\limits_{n\in N}\delta_n\otimes 1)(h\otimes p)=\sum\limits_{n\in N}\delta_n(h)=1,$$
	which ensures that $E\neq (1_S \otimes 1)$.
	\label{ex_concretstrito}
\end{exa}
  
\section{Partial Actions}

\subsection{Global actions}

Let $A$ be a regular multiplier Hopf algebra and $R$ a nondegenerated algebra. We start recalling the definition of a (global) module algebra and some of its properties that we will need in the text. The stated definitions and propositions are from reference \cite{Action}.

\begin{defi} We call $R$ a \textit{left $A$-module algebra} if there exists a linear map
		\begin{eqnarray*}
	\triangleright : A\otimes R &\longrightarrow &R\\
	a\otimes x &\longmapsto & a\triangleright x
\end{eqnarray*}
satisfying: 
		\begin{enumerate}
		\item[(i)] $a\triangleright b\triangleright x=ab\triangleright x,\,\,\,\text{for all}\,\,\, a,b\in A\,\,\,\text{and}\,\,\, x\in R$;
		\item[(ii)] $R$ is \emph{unitary}, that is, $A\triangleright R=R$;
		\item[(iii)] $a\triangleright (xy)=(a_{(1)}\triangleright x)(a_{(2)}\triangleright y)$, for all $x,y\in R$ and $a\in A$.
	\end{enumerate}
	\label{defmodalgebra}
\end{defi}
 In this case, the map $\triangleright$ is called the \emph{action} of $A$ on $R$.

The identity (iii) make sense, since we have that $R$ is a unitary, then, for all $y\in R$, we can write $y=\sum_i b_i\triangleright y_i$, thus
\begin{eqnarray}
	a\triangleright (xy)&=& a\triangleright ( x(\sum_ib_i\triangleright y_i)) \nonumber\\
	&=&\sum_i(a_{(1)}\triangleright x)(a_{(2)}b_i\triangleright y_i) \label{global10} \\ 
	&=& (a_{(1)}\triangleright x)(a_{(2)}\triangleright y), \nonumber
\end{eqnarray}
for every $a\in A$, $x\in R.$

We say that $\triangleright$ is \emph{nondegenerate} if the following holds: $A\triangleright x=0$ if and only if $x=0$. In particular, the action of $A$ on itself via its multiplication is nondegenerate.
\begin{pro} If $R$ is a unitary left $A$-module algebra then the action of $A$ on $R$ is nondegenerate.
	
	\label{acaodegenerada}
\end{pro}
\begin{rem} If $R$ is a unitary left $A$-module algebra, then given $a_1, ..., a_n\in A$ and $x_1, x_2, ..., x_m\in R$ there exists an element $e\in A$ such that $ea_i=a_i=a_ie$, for all $1\leq i\leq n$, and $e\triangleright x_j=x_j$, for all $1\leq j\leq m$.
	\label{obs_unidacao}
\end{rem}

\begin{pro} Assume that $A$ is a regular multiplier Hopf algebra and $R$ is a left $A$-module algebra. Then the action of $A$ on $R$ can be uniquely extended to a nondegenerate  action of $A$ on $M(R)$ as follows:
	\begin{eqnarray*}
		(a\triangleright m)x&=&a_{(1)}\triangleright(m(S(a_{(2)})\triangleright x)),\\
		x(a\triangleright m)&=&  a_{(2)}\triangleright ((S^{-1}(a_{(1)})\triangleright x)m),
	\end{eqnarray*}
	for all $a\in A$, $m\in M(R)$ and $x\in R$. Moreover, $a\triangleright 1_{M(R)}=\varepsilon(a)1_{M(R)}$, for all $a\in A$.
	\label{pro_extacao}
\end{pro}
Nevertheless, in general $M(R)$ is not unitary. 

\vu

\subsection{Partial actions}
Throughout the text partial actions of $A$ on $R$ will be always consider on the left. The partial actions on the right are defined in a similar way. We start the section with the classical definition of partial action, i.e., in the case that $A$ and $R$ are both unital, according to \cite{Caenepeel}.

\begin{defi} \label{part_mod_algHopf} Let $A$ be a Hopf algebra and $R$ a unital algebra. We say that $R$ is a \emph{partial $A$-module algebra} if there exists a linear map $\cdot:A \otimes R \to R$ such that, for all $x,y \in R$ and $a,b \in A$:
		\begin{itemize}
			\item[(i)] $1_A \cdot x = x$;
			\item[(ii)] $a \cdot (xy) = (a_{(1)} \cdot x)(a_{(2)} \cdot y)$;
		\item[(iii)] $a \cdot (b \cdot x) = (a_{(1)} \cdot 1_R)(a_{(2)}b \cdot x)$.
		\end{itemize}
	
	In this case, the linear map $\cdot$ is called a 	\emph{partial action} of $A$ on $R$.
	Furthermore, we say that the partial action is \emph{symmetric} (or, $R$ is a \emph{symmetric partial $A$-module algebra}) if the additional condition also holds:
	
	\begin{itemize}
		\item[(iv)] $a \cdot (b \cdot x) = (a_{(1)}b \cdot x)(a_{(2)} \cdot 1_R)$.
	\end{itemize}
\end{defi}

	Observe that, assuming the condition (i), the conditions (ii) and (iii) in Definition \ref{part_mod_algHopf} are
	equivalent to $$a \cdot (x(b \cdot y)) = (a_{(1)} \cdot x)(a_{(2)}b \cdot y).$$

For the context where both algebras are nonunital, clearly, items (i) and (iii) do not make sense anymore. In addition, the right hand side of the item (ii) is also not well defined for the context of multiplier Hopf algebras. In this way, for extending the notion of partial actions of a regular multiplier Hopf algebra $A$ on an algebra $R$ with a nondegenerate product, we need some extra conditions.

 \begin{defi}\label{def27}
 Let $A$ be a regular multiplier Hopf algebra and $R$ a nondegenerate algebra. A triple $(R, \cdot, \e)$ is a \emph{partial $A$-module algebra} if $\cdot$ is a linear map
 		\begin{eqnarray*}
 			\cdot: A\otimes R & \longrightarrow & R\\
 			a\otimes x & \longmapsto & a\cdot x
 		\end{eqnarray*}
 and $\e$ is a linear map $\e: A\longrightarrow M(R)$ satisfying the following conditions, for all $a,b\in A$ and $x,y\in R$:
 \begin{enumerate}
 	\item[(i)] $a\cdot (x(b\cdot y)) = (a_{(1)}\cdot x)(a_{(2)}b\cdot y)$;
 	
 	\vu
 	
 	\item[(ii)] $\mathfrak{e}(a)(b\cdot x)=a_{(1)}\cdot (S(a_{(2)})b\cdot x)$ and $\mathfrak{e}(A)R\subseteq A\cdot R$;
 	
 	\vu
 	
 	\item[(iii)] given $a_1,...,a_n\in A$ and  $x_1, ..., x_m\in R$ there exists $b\in A$ such that $a_ib=a_i=ba_i$ and $a_i\cdot x_j=a_i\cdot (b\cdot x_j)$, for all $1\leq i\leq n$ and $1\leq j\leq m$;
 	
 	\vu
 	
 	\item[(iv)] $A\cdot x=0$ if and only if $x=0$, that is, $\cdot$ is a \emph{nondegenerate action}.
 \end{enumerate}

 Under these conditions, the map $\cdot$ is called a \emph{partial action} of $A$ on $R$, and we say that it is \emph{symmetric} if the following additional conditions also hold:
 \begin{enumerate}
 \item[(v)] $a\cdot ((b\cdot x)y)=(a_{(1)}b\cdot x)(a_{(2)}\cdot y)$;
 
 \vu
 
 \item[(vi)] $(b\cdot x)\mathfrak{e}(a)=a_{(2)}\cdot(S^{-1}(a_{(1)})b\cdot x)$;
 
 \vu
 
 \item[(vii)]$R \mathfrak{e}(A)\subseteq A \cdot R $,
 \end{enumerate}
 for all $x,y\in R$ and $a,b\in A$.
 \label{def_acaoparcialmulti}
 \end{defi}

\begin{rem} Some considerations must be made about Defintion \ref{def27}:
	\begin{enumerate}
		\item The regularity condition is needed for the symmetry to make sense, since items (ii) and (vi) require that $S(A)=A$.
		
		\item The motivation to define the linear map $\e$ of the item (ii) was inspired on the partial representation theory presented in \cite{representation}. It was necessary to overcome the issue that appears in the expression $(a \cdot 1_{R})$, in Definition \ref{part_mod_algHopf}. We call attention to the fact that we could not directly define the liner map $\mathfrak{e}$ as $\mathfrak{e}(a)(x)=a_{(1)}\cdot (S(a_{(2)}) \cdot x)$, since the right side of the equality does not make sense whereas in the case of multiplier Hopf algebras $\Delta(a)$ does not lie in $A \otimes A$. Recall that the Sweedler notation needs an additional element in $A$ to cover $\Delta(a)$.

		\item A natural condition to be required would be $ A \cdot R = R $. But, with this condition it is not possible to show equivalence between Definitions	\ref{def_acaoparcialmulti} and \ref{part_mod_algHopf} when $A$ and $R$ are both unital. Moreover, $ A \cdot R = R $ does not imply that the partial action is nondegenerate. In this way, the items (ii), (iii) and (iv) are added to the definition aiming to overcome these problems, as we shall see in the following result.
		
	\end{enumerate}	
\end{rem}

\begin{pro}
If $A$ and $R$ are unital algebras, then the Definitions \ref{part_mod_algHopf} and \ref{def_acaoparcialmulti} are equivalent.	
\end{pro}

\begin{proof}
Indeed, Definition \ref{part_mod_algHopf} implies Definition \ref{def_acaoparcialmulti} taking the linear map $\mathfrak{e}: A\longrightarrow M(R)=R$ given by $\mathfrak{e}(a)=a\cdot 1_R$, for all $a\in A$. Conversely, it is enough to check that $1_A\cdot x=x$, for all $x\in R$. To do this take $a, 1_A\in A$ and $x\in R$. By (iii) of Definition \ref{def_acaoparcialmulti} there exists an element $b\in A$ such that $ba=a=ab$, $b1_A=1_A=1_Ab$ and $a\cdot b\cdot x=a\cdot x$. However, $1_A$ is the identity element of $A$, hence $b=1_Ab=1_A$ and $a\cdot 1_A\cdot x=a\cdot x$. Then we have $a\cdot 1_A\cdot x=a\cdot x$, for all $a\in A$ and using (iv) of Definition \ref{def_acaoparcialmulti} we conclude $1_A\cdot x=x$, for all $x\in R$.
\end{proof}

It is immediate to check that any (global) action is a particular example of a partial action with the linear map $\mathfrak{e}: A\longrightarrow M(R)$ defined by  $\e(a)= \varepsilon(a)1_{M(R)}$, for all $a \in A$.
The next proposition characterizes under what condition a partial action is a global one.

\begin{pro}
Assume that  $(R, \cdot, \e)$ is a partial $A$-module algebra. Then $R$ is an $A$-module algebra if and only if $\mathfrak{e}(a)=\varepsilon(a)1_{M(R)}$, for all $a \in A$.
\end{pro}

\begin{proof}
	Suppose that $R$ is an $A$-module algebra, then by (ii) of Definition \ref{def_acaoparcialmulti} we have
	\begin{eqnarray*}
		\mathfrak{e}(a)(b\cdot x)&=&a_{(1)}\cdot(S(a_{(2)})b\cdot x)\\
		&=& a_{(1)}S(a_{(2)})b\cdot x\\
		&=& \varepsilon(a)b\cdot x\\
		&=&\varepsilon(a)1_{M(R)}(b\cdot x),
	\end{eqnarray*}
	for all $a,b\in A$ and $x\in R$. Therefore, as the action is unitary it follows that $\mathfrak{e}(a)x=\varepsilon(a)1_{M(R)}x$, for all $x\in R$ and, consequently, $\mathfrak{e}(a)=\varepsilon(a)1_{M(R)}$, for all $a\in A$.
	
	\vu
	
	Conversely, take $c\in A$ such that $\varepsilon(c)=1_{\Bbbk}$ and $x\in R$. Then
	\begin{eqnarray*}
	a\cdot (b\cdot x)&=& a_{(1)}\cdot(\varepsilon(ca_{(2)})b\cdot x)\nonumber\\
	&=& a_{(1)}\cdot (S(c_{(1)}a_{(2)})c_{(2)}a_{(3)}b\cdot x)\nonumber\\
	&=& a_{(1)}\cdot (S(a_{(2)})S(c_{(1)})c_{(2)}a_{(3)}b\cdot x)\nonumber\\
	&=& a_{(1)}\cdot (S(a_{(2)})\varepsilon(c)a_{(3)}b\cdot x)\nonumber\\
	&=& a_{(1)}\cdot (S(a_{(2)})a_{(3)}b\cdot x)\nonumber\\
	&\stackrel{\ref{def_acaoparcialmulti}(ii)}{=}&\mathfrak{e}(a_{(1)})(a_{(2)}b\cdot x)\nonumber\\
	&=& \varepsilon(a_{(1)})1_{M(R)}(a_{(2)}b\cdot x)\nonumber\\
	&=& ab\cdot x,
	\label{eqcoberturacom_c}
	\end{eqnarray*}
	for all $a,b\in A$. 
	
Furthermore, by condition (ii) of Definition \ref{def_acaoparcialmulti} we have that $x= \varepsilon(c)1_{M(R)}x=\mathfrak{e}(c)x\in A\cdot R$,  for every $x\in R$. Hence, $R$ is unitary. 

\vu

Finally, given $x,y\in R$ with $y= \sum_i b_i\cdot y_i$ (notice that $R=A\cdot R)$, then 
\begin{eqnarray*}
	a\cdot (xy)&=&a\cdot (x(\sum_ib_i\cdot y_i))\\
	&\stackrel{\ref{def_acaoparcialmulti}(i)}{=}& \sum_i (a_{(1)}\cdot x)(a_{(2)}b_i\cdot y_i)\\
	&=& (a_{(1)}\cdot x)(a_{(2)}\cdot y),
\end{eqnarray*}
for all $a\in A$. Therefore, $R$ is an $A$-module algebra.
\end{proof}

\subsection{Examples of partial actions}
The following result give us the necessary and sufficient conditions to provide a family of  examples of partial actions.

\begin{pro}
	Let $A$ be a regular multiplier Hopf algebra, $R$ a nondegenerate algebra and $\lambda: A \longrightarrow \Bbbk$ be a linear map. Then
	\begin{eqnarray*}
		\cdot: A \otimes R & \longrightarrow & R\\
		a \otimes x & \longmapsto & a \cdot x = \lambda(a)x
	\end{eqnarray*}
and $\mathfrak{e}(a)=\lambda(a)1_{M(R)}$ is a partial action of $A$ on $R$ if and only if
	\begin{enumerate}
		\item[(i)] $\lambda(a) \lambda(b)=\lambda(a_{(1)})\lambda(a_{(2)} b)$, for all $a,b\in A$, where $\Delta(a)(1 \otimes b) = a_{(1)} \otimes a_{(2)} b$;
		
		\vu
		
		\item[(ii)] given $a_1, ..., a_n \in A$, there exists $b\in A$ such that $a_i b =a_i=b a_i$ and $\lambda(a_i)\lambda(b)= \lambda(a_i)$,  for all $1\leq i\leq n$.
	\end{enumerate}
	\label{pro_acaolambda}

Furthermore, the conditions of symmetry of Definition \ref{def_acaoparcialmulti} are equivalent to $\lambda(a) \lambda(b)=\lambda(a_{(1)}b)\lambda(a_{(2)})$, for all $a,b\in A$, where $\Delta(a)(b \otimes 1) = a_{(1)}b \otimes a_{(2)} $.
\end{pro}

\begin{proof}
It is immediate.
\end{proof}

\begin{exa}\label{chato}
	Let $R$ be an algebra with a nondegenerate product, $A_G$ the algebra of the linear functions from $G$ to $\Bbbk$ with finite support, and $N$ a finite subgroup of $G$ such that the characteristic of $\Bbbk$ does not divide the order of $N$. Notice that $\{\delta_p\}_{p \in G}$ given by $\delta_p(g)=\delta_{p,g}$ (the Kroneker symbol), for all $g\in G$, is a basis for $A_G$. Define the linear map
	\begin{eqnarray*}
		\lambda: A_G & \longrightarrow & \Bbbk\\
		\delta_g & \longmapsto & \left\{
		\begin{array}{rl}
			\frac{1}{|N|} & \text{if}\quad g\in N\\
			0 & \text{otherwise. }
		\end{array} \right.
	\end{eqnarray*}
	Thus, $R$ is a symmetric partial $A_G$-module algebra with the action given by $\delta_g\cdot x = \lambda(\delta_g)x$ and $\mathfrak{e}(\delta_g)=\lambda(\delta_g)1_{M(R)}$, for all $\delta_g\in A_G$ and $x\in R$.
\end{exa}

For the next example, consider the dual algebra $(\widehat{A}, \widehat{\Delta})$ as defined in (\ref{achapeu}).\\

\begin{exa}\label{exemplochapeu}
	Let $A_G$ and  $R$ be the algebras of the Example \ref{chato} and  $f\in M(A_G)$  defined by
	\begin{eqnarray*}
		f: G & \longrightarrow & \Bbbk\\
		g & \longmapsto & \left\{
		\begin{array}{rl}
			1 & \text{if} \quad g\in N\\
			0 & \text{otherwise},
		\end{array} \right.
	\end{eqnarray*}
	where $N$ is a given subgroup of $G$. Notice that, $f$ is an idempotent element of $M(A_G)$ and $(f\otimes 1)\Delta(f)=f\otimes f$.  Then $R$ is a partial $\widehat{A_G}$-module algebra via the partial action defined by $\varphi(\underline{\hspace{0.3cm}} h) \cdot x = x \varphi(fh)$ and $\mathfrak{e}(\varphi(\underline{\hspace{0.3cm}} h))=\varphi(fh)1_{M(R)}$, for all $x\in R$ and $\varphi(\underline{\hspace{0.3cm}} h)\in \widehat{A_G}$.
\end{exa}

The next proposition extends the notion of induced partial action, presented in \cite{MunizandBatista}, for the context of multiplier Hopf algebras.

 \begin{pro}[Induced Partial Actions]
	Assume that $R$ is an $A$-module algebra via a global action $\triangleright$ and let $L\subset R$ be  a right and unital ideal of $R$ with identity element $1_L$. Then $L$ is a partial $A$-module algebra via $a \cdot x = 1_L(a \triangleright x)$ and $\e(a)=a\cdot 1_L$ for all $a\in A$ and $x\in L$.
	\label{pro_induzida}
\end{pro}

\begin{proof}
	
	Indeed,
	\begin{itemize}
		\item [(i)]	for $a,b\in A$ and $x,y\in L$,
		\begin{eqnarray*}
			a\cdot(x(b\cdot y))&=&1_L(a\triangleright(x1_L(b\triangleright y)))\\
			&=&1_L(a\triangleright(x(b\triangleright y)))\\
			&=&1_L(a_{(1)}\triangleright x)(a_{(2)}b\triangleright y)\\
			&=& 1_L(a_{(1)}\triangleright x)1_L(a_{(2)}b\triangleright y)\\
			&=& (a_{(1)}\cdot x)(a_{(2)}b\cdot y).
		\end{eqnarray*}
		
		\item [(ii)] It follows from the definition of $\e: A\longrightarrow L$.
		
		\vu
		
		\item [(iii)] Given $a_1, ..., a_n\in A$ and $x_1, ..., x_m\in L$, by Remark \ref{obs_unidacao} there is an element $b\in A$ such that $ba_i=a_i=a_ib$ and $b\triangleright x_j=x_j$, thus
		\begin{eqnarray*}
			a_i\cdot b\cdot x_j&=& 1_L(a_i\triangleright (1_L(b\triangleright x_j))\\
			&=& 1_L(a_i\triangleright x_j)\\
			&=& a_i\cdot x_j,
		\end{eqnarray*}
		for all $1\leq i\leq n$ and $1\leq j\leq m$.
		
		\vu
		
		\item [(iv)] If $x\in L$ is such that $a\cdot x=0$, for all $a\in A$, then again by Remark \ref{obs_unidacao} there exists $b\in A$ such that $b\triangleright x=x$ and hence $0=b\cdot x=1_L(b\triangleright x)=1_Lx=x$.
	\end{itemize}
	
\end{proof}

Observe that, if the algebra $L$ is a bilateral ideal of $R$, then its unit $1_L$ is a central idempotent element in $R$. In this case, the induced partial action is symmetric.
 
The following example illustrates the previous proposition.	

\begin{exa} Let $A_G$ be the algebra given by Example \ref{chato} and $R$ the group algebra $\Bbbk G$. Suppose that $R$ is the $A_G$-module algebra via the action $\delta_p\triangleright h=\delta_p(h)h$, for all $p,h\in G$. Consider a finite and normal subgroup $N\neq \{1_G\}$ of $G$, with order $|N|$  not divisible by the characteristic of $\Bbbk$, and $f_N=\frac{1}{|N|}\sum\limits_{n\in N}n$ a central idempotent in $R$. Thus, the algebra $S=f_NR$ is a symmetric partial $A_G$-module algebra given by  
\begin{eqnarray*}
\delta_p\cdot (f_Nh)&=&f_N(\delta_p\triangleright(f_Nh))\\
&=& \left\{
			\begin{array}{rl}
			\frac{1}{|N|} f_Np & \text{if}\quad ph^{-1}\in N\\
				0 \ \ \ \  & \text{otherwise}.
							\end{array} \right.
\end{eqnarray*}
and $\mathfrak{e}(\delta_p)=\delta_p\cdot f_N=\frac{1}{|N|}f_N$. Notice that taking $h=1_G$ and $1_G\neq p\in N$, then 
\begin{center}
$\varepsilon(\delta_p)f_N=\delta_{p,1_G}f_N=0$,
\end{center}
i.e., $\mathfrak{e}(\delta_p)\neq \varepsilon(\delta_p)f_N$. Hence, the induced partial action is not global.
\end{exa}

\subsection{Extensions of a Partial Action}

In this section,  given $(R, \cdot, \e)$ a symmetric partial action of $A$, our purpose is to construct a linear map $\cdot: A \otimes M(R) \rightarrow M(A \cdot R)$. This new linear map will not characterize, however, a structure of partial module algebra on $M(R)$, but it helps, for example, to show that $\e(a)|_{A\cdot R}= a\cdot 1_{M(R)}$. The next result is crucial to define this linear map. 

 \begin{lem}
 	Let $(R, \cdot, \e)$ be a symmetric partial $A$-module algebra. Then
 	\begin{enumerate}
 		\item[(i)] $(a \cdot x)(b \cdot y) = a_{(1)} \cdot (x(S(a_{(2)})b \cdot y))$;
 		
 		\vu
 		
 		\item[(ii)] $(a \cdot x)(b \cdot y) = b_{(2)} \cdot ((S^{-1}(b_{(1)})a \cdot x)y)$,
 	\end{enumerate}
 	for all $a,b\in A$ and $x,y\in R$.
 \end{lem}
\begin{proof}
It is immediate by Definition \ref{def_acaoparcialmulti}.
\end{proof}

For the rest of this section we will assume that the product in $A \cdot R$ is nondegenerate, which, in the global case, follows directly from the fact that $R$ is unitary. 
Notice, in particular,  that Propositions \ref{pro_induzida} and \ref{pro_acaolambda} provide examples of partial actions such that $A\cdot R$ has a nondegenerate product.

 \begin{lem}\label{lema_extparcial} Under these conditions, consider $m\in M(R),$ $a\in A$ and the linear maps
 \begin{eqnarray*}
 \overline{(a \cdot m)}(b\cdot x) &=& a_{(1)} \cdot (m (S(a_{(2)})b \cdot x))\nonumber \\ 
 \overline{\overline{(a \cdot m)}}(b \cdot x) &=& a_{(2)} \cdot ((S^{-1}(a_{(1)})b\cdot x)m),\label{def_exteparcial}
 \end{eqnarray*}
 for all $b\in A$ and $x\in R$. Then $a \cdot m=(\overline{a\cdot m},\overline{\overline{a\cdot m}})\in M(A\cdot R).$ 
  \end{lem}

\begin{proof}
 It is enough to show the compatibility relation between these linear maps. Given $a,b,c\in A$, $x,y\in R$ and $m \in M(R)$,
\begin{eqnarray*}
	(b\cdot x)(\overline{(a\cdot m)}(c\cdot y))&=& (b\cdot x)(a_{(1)}\cdot(m(S(a_{(2)})c\cdot y)))\\
	&\stackrel{\ref{def_acaoparcialmulti}(iv)}{=}&a_{(2)}\cdot((S^{-1}(a_{(1)})b\cdot x)(m(S(a_{(3)})c\cdot y)))\\
	&{=}& a_{(2)}\cdot(((S^{-1}(a_{(1)})b\cdot x)m)(S(a_{(3)})c\cdot y))\\
	&\stackrel{\ref{def_acaoparcialmulti}(i)}{=}& (a_{(2)}\cdot ((S^{-1}(a_{(1)})b\cdot x)m))(a_{(3)}S(a_{(4)})c\cdot y)\\
	&=&   (a_{(2)}\cdot ((S^{-1}(a_{(1)})b\cdot x)m))(c\cdot y)\\
	&=&(\overline{\overline{(a\cdot m)}}(b\cdot x))(c\cdot y).
\end{eqnarray*}
\end{proof}

Then the linear map $\cdot:A\otimes M(R)\to M(A\cdot R)$, as defined in Lemma \ref{lema_extparcial}, is the mentioned extension.
In the next propositions we will see some properties of this extension.

\begin{pro} Let $(R, \cdot, \e)$ be a symmetric partial $A$-module algebra. Then
\begin{enumerate}
\item[(i)] $a \cdot (m(b \cdot n))= (a_{(1)} \cdot m)(a_{(2)}b \cdot n)$;
\item[(ii)] $a\cdot ((b\cdot m)n)=(a_{(1)}b\cdot m)(a_{(2)}\cdot n)$,
\end{enumerate}
for all $a,b\in A$ and $m,n\in M(R)$.
\label{pro_exteproacao}
\end{pro}
\begin{proof}
	(i) In view of Lemma \ref{lema_extparcial},
	\begin{eqnarray*}
		(a \cdot (m(b \cdot n)))(c\cdot x)&=& a_{(1)}\cdot (m(b \cdot n)(S(a_{(2)})c\cdot x))\\
		&=&a_{(1)}\cdot (m((b\cdot n)(S(a_{(2)})c\cdot x)))\\
		&=&a_{(1)}\cdot (m(b_{(1)}\cdot(n(S(b_{(2)})S(a_{(2)})c\cdot x))))\\
&=&a_{(1)}\cdot(m(S(a_{(2)})a_{(3)}b_{(1)}\cdot(n(S(b_{(2)})S(a_{(4)})c\cdot x))))\\	
&=& (a_{(1)} \cdot m)(a_{(2)}b_{(1)}\cdot(n(S(a_{(3)}b_{(2)})c\cdot x)))\\
&=&	(a_{(1)} \cdot m)(a_{(2)}b \cdot n)(c\cdot x),
\end{eqnarray*}
for all $c\in A$ and $x\in R$. Hence,  $(a \cdot (m(b \cdot n)))=(a_{(1)} \cdot m)(a_{(2)}b \cdot n)$, for all $a,b\in A$ and $m,n\in M(R)$. Similarly, one show (ii).
\end{proof}

\begin{rem} \label{obs_e_aigual_a_agindoem1}Notice that $\e(a)|_{A\cdot R}= a\cdot 1_{M(R)}$ for all $a\in A.$ Indeed,
\begin{center}
$(a\cdot 1_{M(R)})(b\cdot x)= a_{(1)}\cdot(1_{M(R)}(S(a_{(2)})b\cdot x))= a_{(1)}\cdot(S(a_{(2)})b\cdot x)=\e(a)(b\cdot x)$,
\end{center}
and conversely $(b\cdot x)(a\cdot 1_{M(R)})= (b\cdot x)\e(a)$, for all $b\in A$ and $x\in R$.
\end{rem}

\begin{pro} Let $(R, \cdot, \e)$  be a symmetric partial  $A$-module algebra. Then
\begin{enumerate}
\item[(i)] $A\cdot R=\e(A)R$;
\item[(ii)] $A\cdot R=R\e(A)$.
\end{enumerate}
\label{pro_igualextensaoacao}
\end{pro}

\begin{proof}
Given $a\in A$ and $x\in R$, it follows from Definition \ref{def_acaoparcialmulti} that,  there is $b\in A$ such that $ab=a=ba$ and $a\cdot x=a\cdot b\cdot x$. Hence,
	\begin{eqnarray*}
		a\cdot x&=& a\cdot (b\cdot x)\\
		&\stackrel{\ref{pro_exteproacao}(i)}{=}&(a_{(1)}\cdot 1_{M(R)})(a_{(2)}b\cdot x)\\
		&\stackrel{\ref{obs_e_aigual_a_agindoem1}}{=}&\e(a_{(1)})(a_{(2)}b\cdot x)\in \e(A)R.
	\end{eqnarray*}
\end{proof}

\subsection{Duality between partial actions and partial coactions}

  In this section, we will establish a duality between partial actions and partial coactions for a regular multiplier Hopf algebra $(A, \Delta)$ with a left integral.

\begin{pro}\label{pro_dualocoacao} Let $(R,\rho,E)$ be a symmetric right partial $A$-comodule algebra. Then  $R$ is a  symmetric left partial $\widehat{A}$-module algebra given by
	\begin{eqnarray*}
		\cdot : \widehat{A}\otimes R &\longrightarrow & R\\
		\varphi(\underline{\hspace{0.3cm}} a)\otimes x&\longmapsto & \varphi(\underline{\hspace{0.3cm}} a)\cdot x := (\imath\otimes\varphi)(\rho(x)(1\otimes a))
	\end{eqnarray*}
and $\e : \widehat{A} \longrightarrow M(R)$ such that
\begin{eqnarray*}
	\e(\varphi(\underline{\hspace{0.3cm}} a))x &=& (\imath\otimes\varphi)(E(x\otimes a)), \\
	x\e(\varphi(b \underline{\hspace{0.3cm}} )) &=& (\imath\otimes\varphi)((x\otimes b)E).
\end{eqnarray*}
	\end{pro}

\begin{proof} Notice that, using the sigma notation, $$\varphi(\underline{\hspace{0.3cm}} a)\cdot x= (\imath\otimes\varphi)(\rho(x)(1\otimes a))=x^{(0)}\varphi(x^{(1)}a),$$ for all $x\in R$ and $a\in A$, where $x^{(0)}\otimes x^{(1)}a\in A\otimes A$. Thus,
	
	\vu
	
	(i) for  $w=\varphi(\underline{\hspace{0.3cm}} a), u=\varphi(\underline{\hspace{0.3cm}} b)\in\widehat{A}$ and $x,y\in R$,
	\begin{eqnarray*}
		w\cdot(x(u\cdot y))	&=& (\imath\otimes\varphi\otimes\varphi)((xy^{(0)})^{(0)}\otimes (xy^{(0)})^{(1)}a\otimes y^{(1)}b)\\
		&=&  (\imath\otimes\varphi\otimes\varphi)((\rho(x)\otimes 1)(\rho\otimes\imath)(\rho(y)(1\otimes b))(1\otimes a\otimes 1))\\
		&\stackrel{(\ref{escritacoa1})}{=}& (\imath\otimes\varphi\otimes\varphi)((\rho(x)\otimes 1)(E\otimes 1)(\imath\otimes\Delta)(\rho(y))(1\otimes a\otimes b))\\
		&\stackrel{\ref{rhoigual_rhoe}}{=}&(\imath\otimes\varphi\otimes\varphi)((\rho(x)\otimes 1)(\imath\otimes\Delta)(\rho(y))(1\otimes a\otimes b)).
	\end{eqnarray*}
		On the other hand, 
		\begin{eqnarray*}
		(w_{(1)}\cdot x)(w_{(2)}u\cdot y)&\stackrel{(\ref{escritadelta1})}{=}& (\varphi(\underline{\hspace{0.3cm}} \ S^{-1}(b_{(1)})a)\cdot x)(\varphi(\underline{\hspace{0.3cm}} \ b_{(2)})\cdot y)\\
		&=&x^{(0)}y^{(0)} \varphi(x^{(1)} S^{-1}(b_{(1)})a) \varphi(y^{(1)} b_{(2)})\\
		&=& (\imath\otimes\varphi)(\rho(x)(y^{(0)}\otimes \varphi(y^{(1)} b_{(2)})S^{-1}(b_{(1)})a))\\
		&=&(\imath\otimes\varphi)(\rho(x)(y^{(0)}\otimes (y^{(1)} b_{(2)})_{(1)}S^{-1}(b_{(1)})a))  \varphi((y^{(1)} b_{(2)})_{(2)})\\
		&=& (\imath\otimes\varphi\otimes\varphi)((\rho(x)\otimes 1)(y^{(0)}\otimes \Delta(y^{(1)}b_{(2)})( S^{-1}(b_{(1)})a\otimes 1 )))\\
		&=&  (\imath\otimes\varphi\otimes\varphi)((\rho(x)\otimes 1)(\imath\otimes \Delta)(\rho(y)(1\otimes b_{(2})))
(1\otimes S^{-1}(b_{(1)})a \otimes 1))\\
		&=& (\imath\otimes\varphi\otimes\varphi)((\rho(x)\otimes 1)(\imath\otimes \Delta)(\rho(y))
(1\otimes  \Delta(b_{(2)})( S^{-1}(b_{(1)})a \otimes 1)))\\
		&=& (\imath\otimes\varphi\otimes\varphi)((\rho(x)\otimes 1)(\imath\otimes\Delta)(\rho(y))(1\otimes a\otimes b)).		
	\end{eqnarray*}
Therefore, $w\cdot(x(u\cdot y))=(w_{(1)}\cdot x)(w_{(2)}u\cdot y)$. 
	
	\vd
	
	(ii) Consider $w=\varphi(\underline{\hspace{0.3cm}} a)$ and $u=\varphi(\underline{\hspace{0.3cm}} b)$ . Hence,
	 \begin{eqnarray*}
		\e(w)(u\cdot x)&=&(\imath\otimes\varphi\otimes\varphi)((E\otimes 1)(\imath\otimes\Delta)(\rho(x))(1\otimes\Delta( b)(a\otimes 1)))\\
		&\stackrel{\ref{def_comoalgparcarcial_multip}}{=}&(\imath\otimes\varphi\otimes\varphi)((\rho\otimes\imath)(\rho(x))(1\otimes\Delta( b)(a\otimes 1)))\\
		&=&  \varphi(\underline{\hspace{0.3cm}} b_{(1)}a)\cdot (\varphi(\underline{\hspace{0.3cm}} b_{(2)})\cdot x)\\
		&\overset{(\ref{escritadelta2})}{=}& w_{(1)}\cdot(\widehat{S}(w_{(2)})u\cdot x).
	\end{eqnarray*}
	
	\vd
	
	(iii) Let $\varphi(\underline{\hspace{0.3cm}} a_1), ..., \varphi(\underline{\hspace{0.3cm}} a_n)\in\widehat{A}$ and $\varphi(c_i \underline{\hspace{0.3cm}})=\varphi(\underline{\hspace{0.3cm}} a_i)$, for all $1\leq i\leq n$. By assumption, $(1\otimes c_i)E=\sum\limits_{j=1}^k m_{ij}\otimes d_{ij}\in M(R)\otimes A$, then we take $\varphi(\underline{\hspace{0.3cm}} b)\in \widehat{A}$ such that $\varphi(\underline{\hspace{0.3cm}} b)\varphi(\underline{\hspace{0.3cm}} a_i)=\varphi(\underline{\hspace{0.3cm}} a_i)=\varphi(\underline{\hspace{0.3cm}} a_i)\varphi(\underline{\hspace{0.3cm}} b)$ and $\varphi(\underline{\hspace{0.3cm}} b)\varphi(d_{ij} \underline{\hspace{0.3cm}})=\varphi(d_{ij} \underline{\hspace{0.3cm}})=\varphi(d_{ij} \underline{\hspace{0.3cm}})\varphi(\underline{\hspace{0.3cm}} b)$, for all  $1\leq i\leq n$ and  $1\leq j\leq k$.
	
	Under the above notation, $a_i\otimes b=\sum\limits_{l=1}^t \Delta(e_{il})(e'_{il}\otimes 1)\in A\otimes A$, for all $1\leq i\leq n$, thus, there exists $f\in A$ such that $fa_i=a_i=a_if$ and $fe_{il}=e_{il}=e_{il}f$, for all  $1\leq i\leq n$ and  $1\leq l\leq t$. Therefore,
	\begin{eqnarray*}
		\varphi(\underline{\hspace{0.3cm}} a_i)\cdot(\varphi(\underline{\hspace{0.3cm}} b)\cdot x)	&=& (\imath\otimes\varphi\otimes\varphi)((\rho\otimes \imath)(\rho(x)(1\otimes b))(1\otimes a_i\otimes 1))\\
		&\stackrel{\ref{def_comoalgparcarcial_multip}}{=}& (\imath\otimes\varphi\otimes\varphi)((E\otimes 1)(\imath\otimes\Delta)(\rho(x))(1\otimes a_i\otimes b))\\
		&=& (\imath\otimes\varphi\otimes\varphi)((E\otimes 1)(\imath\otimes\Delta)(\rho(x))(1\otimes \sum\limits_{l=1}^t \Delta(e_{il})(e'_{il}\otimes 1)))\\
		&=& (\imath\otimes\varphi\otimes\varphi)((E\otimes 1)(\imath\otimes\Delta)(\rho(x))(1\otimes \sum\limits_{l=1}^t  \Delta(fe_{il})(e'_{il}\otimes 1)))\\
		&=& (\imath\otimes\varphi\otimes\varphi)((E\otimes 1)(\imath\otimes\Delta)(\rho(x))(1\otimes \Delta(f)(a_i\otimes b)))\\
		&=& (\imath\otimes\varphi\otimes\varphi)((E\otimes 1)(x^0\otimes \Delta(x^1f)(1\otimes b))(1\otimes a_i\otimes 1))\\
		&=& (\imath\otimes\varphi\otimes\varphi)((1\otimes c_i\otimes 1)(E\otimes 1)(x^{(0)}\otimes \Delta(x^{(1)}f)(1\otimes b))).
	\end{eqnarray*}
		
		On the other hand, 
		\begin{eqnarray*}
		 \varphi(\underline{\hspace{0.3cm}} a_i)\cdot x&=&(\imath\otimes\varphi)(\rho(x)(1\otimes a_i))\\
		&=&(\imath\otimes\varphi)(\rho(x)(1\otimes fa_i))\\
		&=& x^{(0)}\varphi((x^{(1)}f)a_i)\\
		&=& x^{(0)}\varphi(c_i(x^{(1)}f))\\
		&=&  (\imath\otimes\varphi)((1\otimes c_i)E\rho(x)(1\otimes f))\\
		&=& (\imath\otimes\varphi)(\sum\limits_{j=1}^km_{ij} x^{(0)}\otimes d_{ij}x^{(1)}f)\\
		&=&  \sum\limits_{j=1}^km_{ij} x^{(0)} (\varphi(d_{ij}\underline{\hspace{0.3cm}} ))(x^{(1)}f)\\
		&=& \sum\limits_{j=1}^k m_{ij} x^{(0)} (\varphi(d_{ij}\underline{\hspace{0.3cm}} )\varphi(\underline{\hspace{0.3cm}} b))(x^{(1)}f)\\
		&=& \sum\limits_{j=1}^km_{ij} x^{(0)}\varphi(d_{ij}(x^{(1)}f)_{(1)})\varphi((x^{(1)}f)_{(2)}b)\\
		&=&(\imath\otimes\varphi\otimes\varphi)((\sum\limits_{j=1}^k m_{ij}\otimes d_{ij}\otimes 1)(x^{(0)}\otimes (x^{(1)}f)_{(1)} \otimes (x^{(1)}f)_{(2)}b))\\
		&=& (\imath\otimes\varphi\otimes\varphi)((1\otimes c_i\otimes 1)(E\otimes 1)(x^{(0)}\otimes \Delta(x^{(1)}f)(1\otimes b))),
	\end{eqnarray*}
	that is, $\varphi(\underline{\hspace{0.3cm}} a_i)\cdot(\varphi(\underline{\hspace{0.3cm}} b)\cdot x)=\varphi(\underline{\hspace{0.3cm}} a_i)\cdot x$, for all $1\leq i\leq n$ and $x \in R$.
	
	 \vd
	(iv) Suppose that $w \cdot x =0$, for all $w \in\widehat{A}$, we will show that $x =0$. We know that, in particular, $\varphi(a\underline{\hspace{0.3cm}} b)\cdot x = 0$, for any $a,b\in A$. Then
	\begin{eqnarray*}
		0&=& \varphi(a\underline{\hspace{0.3cm}} b)\cdot x\\
		&=& (\imath\otimes\varphi)((1\otimes a)\rho(x)(1\otimes b))\\
		&=& (\imath\otimes\varphi)((1\otimes a)(\sum_{i} y_i \otimes b_i)))\\
		&=& \sum_i y_i\varphi(ab_i),
	\end{eqnarray*}
where the set $\{y_i\}$ is linearly independent. Hence $\varphi(ab_i)=0$, for all $i$ and $a\in A$, what implies $b_i=0$ for each $i$. In this way $\rho(x)(1\otimes b)=0$ for every $b \in A$, thus, by the injectivity of $\rho$, $x=0$.
\end{proof}

\begin{exa}Let $R$ be a partial $A_G$-comodule algebra given by Example \ref{ex_mlambda} where $\rho(x)=x\otimes m$ for all $x\in R$. In this case, $m\in M(A_G)$ such that $m(g)=1$ if $g\in N$ and $m(g)=0$ otherwise, for $N$ any subgroup of $G$ and besides that $E=1\otimes m.$ Therefore, $R$ is a partial $\widehat{A}_G$-module algebra via 
	 \begin{eqnarray*}
	\varphi(\underline{\hspace{0.3cm}} f)\cdot x&=& x\varphi(mf)\\
	&=& x \sum_{g\in N} f(g),
	\end{eqnarray*}
and $\e(\varphi(\underline{\hspace{0.3cm}} f))=(\imath\otimes\varphi)(E(1\otimes f))=\sum_{g\in N} f(g).$
	
	Note that, $\e(\varphi(\underline{\hspace{0.3cm}} f))x=\varphi(\underline{\hspace{0.3cm}} f)\cdot x$, for all $x\in R$ and $\varphi(\underline{\hspace{0.3cm}} f)\in \widehat{A}_G.$ Recall that $\widehat{A}_G \cong \Bbbk G$, then, this partial action is related to a partial group action on $R$.
\end{exa}

To show the converse of the above result we will need to assume some extra conditions. 

\begin{pro} \label{propfeia} Consider $(R, \cdot, \e)$ a symmetric left partial  $A$-module algebra, such that $A\cdot R$ has nondegenerate product. If there exist a fixed element $b\in A$ and a linear map $f: A\longrightarrow M(R)$ such that $\e=f(\underline{\hspace{0.3cm}} b)$, and 
	\begin{enumerate}
		\item[(i)] $\e(a_{(1)})\e(a_{(2)})=\e(a)$ for all $a\in A$;
		\item[(ii)]$ \e(k)=1_{M(A\cdot R)}$, where $kb=b=bk$,
	\end{enumerate}
 then,  $(A\cdot R, \rho, E)$ is a symmetric right partial $\widehat{A}$-comodule algebra with $\rho: A\cdot R\longrightarrow M((A\cdot R)\otimes\widehat{A})$ given by
 \begin{eqnarray*}
 	\rho(a\cdot x)(1\otimes\varphi(\underline{\hspace{0.3cm}} b))&=& \e(S^{-1}(b_{(2)}))(S^{-1}(b_{(1)})a\cdot x)\otimes\varphi(\underline{\hspace{0.3cm}} b_{(3)})\\
 	(1\otimes\psi(\underline{\hspace{0.3cm}} b))\rho(a\cdot x)&=& (S(b_{(3)})a\cdot x)\e (S(b_{(2)}))\otimes \psi(\underline{\hspace{0.3cm}} b_{(1)}),
 \end{eqnarray*}
 and, $E\in M((A\cdot R)\otimes\widehat{A})$ such that 
 $$E(1\otimes \varphi(\underline{\hspace{0.3cm}} b))=\e(S^{-1}(b_{(1)}))|_{A\cdot R}\otimes\varphi(\underline{\hspace{0.3cm}} b_{(2)})$$ $$(1\otimes\psi(\underline{\hspace{0.3cm}} b))E=\e(S(b_{(2)}))|_{A\cdot R}\otimes\psi(\underline{\hspace{0.3cm}} b_{(1)}),$$ 
 for all $\varphi(\underline{\hspace{0.3cm}} b)$ and $\psi(\underline{\hspace{0.3cm}} b)\in\widehat{A}$.
 
\end{pro}
 \begin{proof}
  	Note that, since $\e=f(\underline{\hspace{0.3cm}} b)$, $E$ and $\rho$ are well defined. It is not hard to prove that $E$ and $\rho(a \cdot x)$ lies in $ M((A\cdot R)\otimes\widehat{A})$, for all $a \cdot x \in A \cdot R$. Moreover:
 
 $\bullet $ $E^2=E$.
 	\begin{eqnarray*}
 		E(E(a\cdot x\otimes\varphi(\underline{\hspace{0.3cm}} b)))&=& E(\e(S^{-1}(b_{(1)}))(a\cdot x)\otimes\varphi(\underline{\hspace{0.3cm}} b_{(2)}))\\
 		&=&\e(S^{-1}(b_{(2)}))\e(S^{-1}(b_{(1)}))(a\cdot x)\otimes\varphi(\underline{\hspace{0.3cm}} b_{(3)})\\
 		&\stackrel{(i)}{=}& \e(S^{-1}(b_{(1})))(a\cdot x)\otimes\varphi(\underline{\hspace{0.3cm}} b_{(2)})\\
 		&=& E(a\cdot x\otimes\varphi(\underline{\hspace{0.3cm}} b)),
 	 	\end{eqnarray*}
 		for all $a\cdot x\in A\cdot R$ and $\varphi(\underline{\hspace{0.3cm}} b)\in \widehat{A}$.

 	$\bullet $ $(\rho\otimes\imath)(\rho(a\cdot x)(1\otimes\varphi(\underline{\hspace{0.3cm}} b)))(1\otimes\varphi(\underline{\hspace{0.3cm}} c)\otimes 1)=(E\otimes 1)(\imath\otimes\widehat{\Delta})(\rho(a\cdot x))(1\otimes\varphi(\underline{\hspace{0.3cm}} c)\otimes\varphi(\underline{\hspace{0.3cm}} b))$, for all $a\cdot x\in\widehat{A}\cdot R$ and $\varphi(\underline{\hspace{0.3cm}} b),\varphi(\underline{\hspace{0.3cm}} c)\in\widehat{A}$.
 	{\small
 		\begin{eqnarray*}
 			&& \hspace{-1cm}((\rho\otimes\imath)(\rho(a\cdot x)(1\otimes\varphi(\underline{\hspace{0.3cm}} b)))(1\otimes\varphi(\underline{\hspace{0.3cm}} c)\otimes 1))(1\otimes d\otimes g)\\
 		&=&(\rho(\e(S^{-1}(b_{(2)}))(S^{-1}(b_{(1)})a\cdot x))(1\otimes\varphi(\underline{\hspace{0.3cm}} c))\otimes  \varphi(\underline{\hspace{0.3cm}} b_{(3)}))(1\otimes d\otimes g)\\
 			&=& (\rho(\e(S^{-1}(b_{(2)})k)(S^{-1}(b_{(1)})a\cdot x))(1\otimes\varphi(\underline{\hspace{0.3cm}} c))\otimes  \varphi(\underline{\hspace{0.3cm}} b_{(3)}))(1\otimes d\otimes g)\\
 		&\stackrel{\ref{def_acaoparcialmulti}}{=}& (\rho(S^{-1}(b_{(1)})k_{(1)}\cdot(S(k_{(2)})a\cdot x))(1\otimes\varphi(\underline{\hspace{0.3cm}} c))\otimes \varphi(\underline{\hspace{0.3cm}} b_{(2)}))(1\otimes d\otimes g)\\
 		&=&(\e(S^{-1}(c_{(3)}))\e(S^{-1}(c_{(2)})S^{-1}(b_{(2)})k)(S^{-1}(c_{(1)})S^{-1}(b_{(1)})a\cdot x)\otimes \varphi(\underline{\hspace{0.3cm}} c_{(4)})\otimes\varphi(\underline{\hspace{0.3cm}} b_{(3)}))(1\otimes d\otimes g)\\
 			&=& \e(S^{-1}(c_{(3)}))\e(S^{-1}(c_{(2)})S^{-1} (b_{(2)}))(S^{-1}(c_{(1)})S^{-1}(b_{(1)})a\cdot x)\varphi(dc_{(4)})\varphi(gb_{(3)})\\
 			&\stackrel{(*)}{=}& \e(\varphi(d_{(2)}c_{(5)})d_{(1)}c_{(4)}S^{-1}(c_{(3)}))\e(S^{-1}(c_{(2)})
\varphi(g_{(2)}b_{(4)})g_{(1)}b_{(3)}S^{-1}(b_{(2)}))(S^{-1}(c_{(1)})S^{-1}(b_{(1)})a\cdot x)\\
 			&=&\e(d_{(1)})\varphi(d_{(2)}c_{(3)})\e(S^{-1}(c_{(2)})g_{(1)})\varphi(g_{(2)}b_{(2)})(S^{-1}(c_{(1)})S^{-1}(b_{(1)})a\cdot x)\\
 			&\stackrel{(**)}{=}& \e(d_{(1)})\e(\varphi(d_{(3)}c_{(4)})d_{(2)}c_{(3)}S^{-1}(c_{(2)})g_{(1)})(S^{-1}(c_{(1)})
\varphi(g_{(3)}b_{(3)})g_{(2)}b_{(2)}S^{-1}(b_{(1)})a\cdot x)\\
 			&=& \e(d_{(1)})\e(d_{(2)}g_{(1)})\varphi(d_{(3)}c_{(2)})(S^{-1}(c_{(1)})g_{(2)}a\cdot x)\varphi(g_{(3)}b)\\
 			&\stackrel{(***)}{=}&  \e(d_{(1)})\e(d_{(2)}g_{(1)})(\varphi(d_{(4)}c_{(3)})d_{(3)}c_{(2)}S^{-1}(c_{(1)})g_{(2)}a\cdot x)\varphi(g_{(3)}b)\\
 			&=&\e(d_{(1)})\e(d_{(2)}g_{(1)})(d_{(3)}g_{(2)}a\cdot x)\varphi(d_{(4)}c)\varphi(g_{(3)}b),
 		\end{eqnarray*} 		
 	}
in which in equalities $(*)$, $(**)$ and $(***)$ we used the left invariance of the integral $\varphi$.
 	
On the other side,
{ \small
\begin{eqnarray*}
 		\vspace{-1cm}& &((E\otimes 1)(\imath\otimes\widehat{\Delta})(\rho(a\cdot x))(1\otimes\varphi(\underline{\hspace{0.3cm}} c)\otimes\varphi(\underline{\hspace{0.3cm}} b)))(1\otimes d\otimes g)=\\
 		&\stackrel{(*)}{=}& ((E\otimes 1)(\imath\otimes\widehat{\Delta})(\rho(a\cdot x))(\imath\otimes\widehat{\Delta})(1\otimes \varphi(\underline{\hspace{0.3cm}} b_{(1)}c))(1\otimes 1\otimes \varphi(\underline{\hspace{0.3cm}} b_{(2)})))(1\otimes d\otimes g)\\
 	&=&((E\otimes 1)(\imath\otimes\widehat{\Delta})(\e(S^{-1}(b_{(2)}c_{(2)}))(S^{-1}(b_{(1)}c_{(1)})a\cdot x)\otimes\varphi(\underline{\hspace{0.3cm}} b_{(3)}c_{(3)}))(1\otimes 1\otimes \varphi(\underline{\hspace{0.3cm}} b_{(4)})))(1\otimes d\otimes g)\\
 		&=&((E\otimes 1)(\e(S^{-1}(b_{(2)}c_{(2)}))(S^{-1}(b_{(1)}c_{(1)})a\cdot x)\otimes
\widehat{\Delta}(\varphi(\underline{\hspace{0.3cm}} b_{(3)}c_{(3)}))(1\otimes \varphi(\underline{\hspace{0.3cm}} b_{(4)}))))(1\otimes d\otimes g)\\
 		&{=}&((E\otimes 1)(\e(S^{-1}(b_{(2)}c_{(2)}))(S^{-1}(b_{(1)}c_{(1)})a\cdot x)\otimes \varphi(\underline{\hspace{0.3cm}} c_{(3)})\otimes \varphi(\underline{\hspace{0.3cm}} b_{(3)})))(1\otimes d\otimes g)\\
 		&=& (\e(S^{-1}(c_{(3)}))\e(S^{-1}(b_{(2)}c_{(2)}))(S^{-1}(b_{(1)}c_{(1)})a\cdot x)\otimes\varphi(\underline{\hspace{0.3cm}} c_{(4)})\otimes\varphi(\underline{\hspace{0.3cm}} b_{(3)}))(1\otimes d\otimes g)\\
 		&=&\e(S^{-1}(c_{(3)}))\e(S^{-1}(b_{(2)}c_{(2)}))(S^{-1}(b_{(1)}c_{(1)})a\cdot x)\varphi(d c_{(4)})\varphi(g b_{(3)})\\
 		&\stackrel{(**)}{=}& \e(\varphi(d_{(2)} c_{(5)})d_{(1)} c_{(4)}S^{-1}(c_{(3)}))\e(S^{-1}(c_{(2)})
\varphi(g_{(2)} b_{(4)})g_{(1)} b_{(3)}S^{-1}(b_{(2)}))(S^{-1}(b_{(1)}c_{(1)})a\cdot x)\\
 		&=&\e(d_{(1)})\varphi(d_{(2)}c_{(3)})\e(S^{-1}(c_{(2)})g_{(1)})\varphi(g_{(2)}b_{(2)})(S^{-1}(c_{(1)})S^{-1}(b_{(1)})a\cdot x)\\
 	&\stackrel{(***)}=& \e(d_{(1)})\e(d_{(2)}g_{(1)})(d_{(3)}g_{(2)}a\cdot x)\varphi(d_{(4)}c)\varphi(g_{(3)}b),
 	\end{eqnarray*}
 	}
in $(*)$ we used $ 1\otimes\varphi(\underline{\hspace{0.3cm}} c)\otimes\varphi(\underline{\hspace{0.3cm}} b)=(\imath\otimes\widehat{\Delta})(1\otimes\varphi(\underline{\hspace{0.3cm}} b_{(1)}c))(1\otimes 1\otimes\varphi(\underline{\hspace{0.3cm}} b_{(2)}))$ and in $(**)$ and $(***)$ we used the left invariance of the integral $\varphi$. Similarly, the symmetry property can be proved.

 	$\bullet $ $\rho$ is homomorphism. Let $a\cdot x, b\cdot y\in A\cdot R$, $d\in A$ and $\varphi(\underline{\hspace{0.3cm}} c)\in\widehat{A}$. Then 
 	\begin{eqnarray*}
 		&& \hspace{-1.5cm}(\rho((a\cdot x)(b\cdot y))(1\otimes\varphi(\underline{\hspace{0.3cm}} c)))(1\otimes d)=\\
 		&=& (\rho(a_{(1)}\cdot(x(S(a_{(2)})b\cdot y)))(1\otimes\varphi(\underline{\hspace{0.3cm}} c)))(1\otimes d)\\
 		&=&\e(S^{-1}(c_{(3)}))(S^{-1}(c_{(2)})a\cdot x)(S^{-1}(c_{(1)})b\cdot y)\varphi(d c_{(4)})\\
 		&\stackrel{(*)}{=}& \e(d_{(1)})(d_{(2)}a\cdot x)(d_{(3)}b\cdot y)\varphi(d_{(4)}c)\\
	&\stackrel{(i)}{=}& \e(d_{(1)})\e(d_{(2)})(d_{(3)}a\cdot x)(d_{(4)}b\cdot y)\varphi(d_{(5)}c)\\
 	&\stackrel{(**)}{=}& \e(d_{(1)})(d_{(2)}a\cdot x)\e(d_{(3)})(d_{(4)}b\cdot y)\varphi(d_{(5)}c)\\
 	&=& \e(d_{(1)})(d_{(2)}a\cdot x)\e(d_{(3)})(\varphi(d_{(5)}c_{(3)})d_{(4)}c_{(2)}S^{-1}(c_{(1)})b\cdot y)\\
 	&\stackrel{(***)}{=}&( \e(S^{-1}(c_{(4)}))(S^{-1}(c_{(3)})a\cdot x)\e(S^{-1}(c_{(2)}))(S^{-1}(c_{(1)})b\cdot y)\otimes \varphi(\underline{\hspace{0.3cm}} c_{(5)}))(1\otimes d)\\
 &=& (\rho(a\cdot x)(\e(S^{-1}(c_{(2)}))(S^{-1}(c_{(1)})b\cdot y)\otimes\varphi(\underline{\hspace{0.3cm}} c_{(3)})))(1\otimes d)\\
  	& =&(\rho(a\cdot x)(\rho(b\cdot y)(1\otimes\varphi(\underline{\hspace{0.3cm}} c))))(1\otimes d),
 	 \end{eqnarray*}
 	
 in $(*)$ and in $(***)$ we used repeatedly the left invariance of the integral $\varphi$ as we did in the previously in the verification of the coassociativity of the coaction, in $(**)$ the equality $e(d_{(1)})(d_{(2)}a\cdot x)=(d_{(1)}a\cdot x)e(d_{(2)})$, corresponding to the symmetry of the partial action was used.
	
 	$\bullet$ $\rho$ is injective. It is enough to verify that $(\imath\otimes\widehat{\varepsilon})(\rho(a\cdot x))=a\cdot x$, where $\widehat{\varepsilon}(\varphi(\underline{\hspace{0.3cm}} a))=\varphi(a)$, for all $a\in A$. Consider $\varphi(\underline{\hspace{0.3cm}} b)\in\widehat{A}$ such that $\widehat{\varepsilon}(\varphi(\underline{\hspace{0.3cm}} b))=1_{\Bbbk}$, then 
 	\begin{eqnarray*}
 		(\imath\otimes\widehat{\varepsilon})(\rho(a\cdot x))
 		&=& (\imath\otimes\widehat{\varepsilon})(\rho(a\cdot x)(1\otimes \varphi(\_ b)))\\
 		&=& (\imath\otimes\widehat{\varepsilon})(\e(S^{-1}(b_{(2)}))(S^{-1}(b_{(1)})a\cdot x)\otimes\varphi(\_ b_{(3)}))\\
 		&=&  \e(S^{-1}(b_{(2)}))(S^{-1}(b_{(1)})a\cdot x)\varphi(b_{(3)})\\
 		&=& \e(S^{-1}(b_{(2)})k)(S^{-1}(b_{(1)})a\cdot x)\varphi(b_{(3)})\\
 		&=& \e(S^{-1}(S(k)b_{(2)}))(S^{-1}(b_{(1)})a\cdot x)\varphi(b_{(3)})\\
 		&=&  \e(k)(S^{-1}(b_{(1)})a\cdot x)\varphi(b_{(2)})\\
 		&\stackrel{(ii)}{=}& (S^{-1}(b_{(1)})S^{-1}(S(a))\cdot x)\varphi(b_{(2)})\\
 		&=& a\cdot x.
 	\end{eqnarray*}
 	Therefore, $A\cdot R$ is a symmetric partial  $\widehat{A}$-comodule algebra.
 \end{proof}
 
\begin{rem} Note that the conditions presented in Proposition \ref{propfeia} are clearly satisfied for the cases of a partial action of a Hopf algebra and for the global action of a multiplier Hopf algebra.  
\end{rem}

\begin{exa}Let $(\alpha_g, R_g)$ be a partial group action from $G$ on  algebra $R$ with nondegenerate product. If each $R_g$ is unital then $R$ is a partial $\Bbbk G$-module algebra, where $\Bbbk G$ is the group algebra, with $\delta_g\cdot x=\alpha_g(x1_{g^{-1}}),$ and $\e(\delta_g)=1_g$ for all $g\in G.$ 
	Notice that, $\e$ satisfies the hypothesis of Proposition \ref{propfeia}. Thus, by $\widehat{\Bbbk G}\cong A_G$ and  $\Bbbk G\cdot R=R$, $R$ is a symmetric right partial $ A_G$-comodule algebra with
	\begin{eqnarray*}
		\rho(x)(1\otimes\varphi(\underline{\hspace{0.3cm}} a))&=&\e(S^{-1}(a_{(2)}))(S^{-1}(a_{(1)})\cdot x)\otimes\varphi(\underline{\hspace{0.3cm}} a_{(3)})\\&=& \sum_{g\in G}a_g\e(\delta_{g^{-1}})(\delta_{g^{-1}}\cdot x)\otimes\varphi(\underline{\hspace{0.3cm}} \delta_g)\\
		&=& \sum_{g\in G}a_g\alpha_g(x1_{g^{-1}})\otimes \varphi(\underline{\hspace{0.3cm}} \delta_g)
	\end{eqnarray*}
	and $E(1\otimes \varphi(\underline{\hspace{0.3cm}} a) )=\sum_{g\in G}a_g1_{g^{-1}}\otimes\varphi(\underline{\hspace{0.3cm}} \delta_g)$, for all $\varphi(\underline{\hspace{0.3cm}} a)=\sum_{g\in G}a_g\varphi(\underline{\hspace{0.3cm}} \delta_g)\in \widehat{\Bbbk G}$.
	\end{exa}

\begin{exa}
	Consider the induced partial action given by Proposition \ref{pro_induzida}, where $L$ is a symmetric partial $A$-module algebra via $a \cdot x = 1_L(a \triangleright x),$ for all $a\in A$ and $x\in L$, with $\e(a)=a\cdot 1_L$, for all $a\in A$.
	
	In this case, $A\cdot L$  has nondegenerate product, because $\triangleright$ is a global action of $A$ on $R$. Moreover, there
	exists $b\in A$ such that $b\triangleright 1_L=1_L$, since $1_L$ is a  central idempotent in $R=A\triangleright R$. Then 
	$$\e(a)=1_L(a\triangleright 1_L)=1_L(a\triangleright (b\triangleright1_L))=1_L(ab\triangleright 1_L)=\e(ab),$$
	for all $a\in A$. Similarly,  $\e(a_{(1)})\e(a_{(2)})=\e(a)$, for all $a\in A$. Besides that, considering the element $k\in A$ such that, $kb=b=bk$, we have that $\e(k)=1_{M(A\cdot L)}$. 
	
	Therefore, $A\cdot L$ is a symmetric right partial $\widehat{A}$-comodule algebra.
\end{exa}

\begin{cor}
	Let $(A\cdot R, \rho, E)$ be a partial $\widehat{A}$-comodule algebra derived by Proposition \ref{propfeia}, then there is a structure of a partial $A$-module algebra on $A \cdot R$ which coincides with the original partial $A$-module algebra $(A \cdot R, \cdot, \e)$.  
\end{cor}

\begin{proof}
Let $A\cdot R$ be a partial $\widehat{A}$-comodule algebra derived by Proposition \ref{propfeia}, i. e., $\rho(a \cdot x)(1 \otimes \varphi(\underline{\hspace{0.3cm}} b)) = \e(S^{-1}(b_{(2)}))(S^{-1}(b_{(1)})a \cdot x) \otimes \varphi(\underline{\hspace{0.3cm}} b_{(3)})$ and $E(1 \otimes \varphi(\underline{\hspace{0.3cm}}  b))=\e(S^{-1}(b_{(1)}))|_{A \cdot R} \otimes \varphi(\underline{\hspace{0.3cm}}  b_{(2)})$. Thus, by Proposition \ref{pro_dualocoacao}, $\widehat{\widehat{A} \ }$ is a partial $(A \cdot R)$-module algebra with
$$\widehat{\widehat{b}}  \rightharpoonup (a \cdot x) = (\imath \otimes \widehat{\psi})(\rho(a \cdot x)(1 \otimes \varphi(\underline{\hspace{0.3cm}} S(b)))) $$
$$\widetilde{\e}(\widehat{\widehat{b}})(a \cdot x)= (\imath \otimes \widehat{\psi} )(E(a \cdot x \otimes \varphi(\underline{\hspace{0.3cm}}  S(b))) ),$$
where $\widehat{\widehat{b}}= \widehat{\psi}(\underline{\hspace{0.3cm}} \varphi(\underline{\hspace{0.3cm}}  S(b))) $.
Indeed, for all $a,b \in A$ and $x \in R$,
\begin{eqnarray*}
\widehat{\widehat{b}}  \rightharpoonup (a \cdot x) &=& (\imath \otimes \widehat{\psi})(\rho(a \cdot x)(1 \otimes \varphi(\underline{\hspace{0.3cm}} S(b))))\\
&=&\e(b_{(2)})(b_{(3)}a \cdot x) \widehat{\psi}(\varphi (\underline{\hspace{0.3cm}}S(b_{(1)})))\\
&=&\e(b_{(1)})(b_{(2)}a \cdot x)\\
&=& b \cdot (a \cdot x),
\end{eqnarray*}
in which we used in the third equality the identity $\widehat{\psi}(\varphi (\underline{\hspace{0.3cm}}S(b)))=\varepsilon(b)$ (\cite{Frame}, Proposition 4.8). Moreover,
\begin{eqnarray*}
\widetilde{\e}(\widehat{\widehat{b}})(a \cdot x)&=& (\imath \otimes \widehat{\psi} )(E(a \cdot x \otimes \varphi(\underline{\hspace{0.3cm}}  S(b))) )\\
&=& (\imath \otimes \widehat{\psi} )(\e(S^{-1}(S(b_{(2)})))(a\cdot x) \otimes \varphi(\underline{\hspace{0.3cm}}  S(b_{(1)})))\\
&=& \e(b_{(2)})(a\cdot x) \widehat{\psi}(\varphi(\underline{\hspace{0.3cm}}  S(b_{(1)})))\\
&=&\e(b)(a \cdot x).
\end{eqnarray*}

Therefore, $A$ is a partial $(A \cdot R)$-module algebra via $\cdot$ .
\end{proof}


\section{Morita Context} In reference \cite{Galois} the authors constructed a Morita context connecting the smash product algebra and the coinvariant algebra. Generalizing these ideas, we extend this result to the setting of partial (co)actions of multiplier Hopf algebras.
\subsection{Smash product algebra and the coinvariant algebra} We start defining the smash product and the algebra of (co)invariant elements. We also present their respective properties which are fundamental for the construction of a generalized Morita context.
\begin{defi} Let $(R, \cdot, \e)$ be a partial $A$-module algebra. The smash product algebra $R\# A$ is the vector space $R\otimes A$ endowed with the product given by the following rule
	\begin{eqnarray*}
		(x\# a)(y\# b)=x(a_{(1)}\cdot y)\# a_{(2)}b,
	\end{eqnarray*}
for all	$x,y \in R$ and $a,b\in A$.
	\label{def_smash}
\end{defi}

Notice that the smash product makes sense because $\Delta(a)(1\otimes b)=a_{(1)}\otimes a_{(2)}b\in A\otimes A$.

\begin{pro}If $(R, \cdot, \e)$ is a partial $A$-module algebra, then the product of $R\# A$ is left nondegenerate, i.e., if $(x\# a)(y\# b)=0$, for all $(x\# a)\in R\# A$, then, $(y\# b)=0$.
\end{pro}
\begin{proof}
	First of all, we observe that any element of $R\# A$ can be written in the form $\sum\limits_{i=1}^{n} y_i \# b_i $ with the $b_i$'s linearly independent. Assuming that $(x\# a)(\sum\limits_{i=1}^ny_i\# b_i)=0$ for all $x\in R$ and $a\in A$, we need to prove that $\sum\limits_{i=1}^ny_i\# b_i=0$. It is enough to prove that $y_i=0$, for all $i \in {1,...,n}$. Indeed, it follows from the nondegeneracy of the product of $R$ that $\sum\limits_{i=1}^n(a_{(1)}\cdot y_i)\# a_{(2)}b_i=0 $.
	
	Now, considering that  for any $c\in A$ there exist $d,e\in A$ such that $(1 \otimes c)\Delta(a)=d \otimes e$,
	\begin{eqnarray*}
		0 &=& (a_{(1)}\cdot y_i)\# ca_{(2)}b_i\\
		&=& \sum_{i=1}^{n} d\cdot y_i\# eb_i.   
		\end{eqnarray*}
	 Since the product of $A$ is nondegenerate $\sum_{i=1}^{n} d\cdot y_i\# b_i=0$. As the $b_i$'s are linearly independent, it follows that $f(d \cdot y_i)=0$, for all $d \in A$, $i \in \{1,...,n\}$ and any linear functional $f$  of $R$. Hence $d\cdot y_i=0$ for all $d\in A$ and so, because of the nondegeneracy of the action (condition (iv) of Definition \ref{def27}),  $y_i=0$, for all $1 \leq i \leq n$.
\end{proof}

\begin{rem} Suppose $(R, \cdot, \e)$ be a symmetric partial $A$-module algebra. If $A\cdot R$ has a nondegenerate product and $m\in M(R)$, then $m|_{A\cdot R}=(\overline{m},\overline{\overline{m}})\in M(A\cdot R)$ as follows 	
		\begin{eqnarray*}
		\overline{m}: A\cdot R  \longrightarrow  A\cdot R  \  \ \ &\mbox{and}& \ \  \overline{\overline{m}}: A\cdot R  \longrightarrow  A\cdot R\\
		x\e(a)  \longmapsto  mx\e(a)  & \  & \ \ \	\ \ \ \ \e(a)x \longmapsto  \e(a)xm.\qquad  
	\end{eqnarray*}
	The well definition of these maps are ensured by Proposition \ref{pro_igualextensaoacao}.
\end{rem}

In what follows, $A\cdot R$ will be an algebra with a nondegenerate product.

\begin{defi}\label{def_invaparcial} Let $(R, \cdot, \e)$ be a symmetric partial $A$-module algebra. We define the \emph{invariant algebra} as the subalgebra of the elements of $M(R)$ which are invariant by the partial action, as follows
	\begin{eqnarray*}
		R^{\underline{A}}=\{m\in M(R); \ a\cdot m=m|_{A\cdot R}(a\cdot 1_{M(R)}) \ \mbox{and} \ a\cdot m=(a\cdot 1_{M(R)})m|_{A\cdot R}, \forall a\in A\}.
	\end{eqnarray*}
	\end{defi}
We denote the invariant subalgebra by $R^{\underline{A}}$ following the notation given in \cite{Galois}.
\begin{pro} If $(R, \cdot, \e)$ is a symmetric partial $A$-module algebra, then
	\begin{enumerate}
		\item[(i)] $\{ m\in M(R); \ a\cdot(xm)=(a\cdot x)m \ \mbox{and} \ a\cdot(mx)=m(a\cdot x), \forall a\in A \ \mbox{and}\ x\in R\} \subseteq R^{\underline{A}}$;
		\item[(ii)] $R^{\underline{A}}\subseteq \{m\in M(R); \ a\cdot(m(c\cdot x))=m(a\cdot x), \ ac=a=ca, \ a\cdot x=a \cdot c \cdot x, \ \forall a\in A \ \  x\in R\}$.
	\end{enumerate}
	\label{carac_invari}
\end{pro}
\begin{proof} (i) Let $m\in M(R)$ such that $a\cdot(xm)=(a\cdot x)m$ and $a\cdot(mx)=m(a\cdot x),$ for all $a\in A$ and $x\in R$, thus
		\begin{eqnarray*}
		(a\cdot m)(b\cdot x)&\stackrel{\ref{lema_extparcial}}{=}& a_{(1)}\cdot(m(S(a_{(2)})b\cdot x))\\
		&=& m(a_{(1)}\cdot(S(a_{(2)})b\cdot x))\\
		&=& m|_{A\cdot R}(a_1\cdot(S(a_2)b\cdot x))\\
		&=&m|_{A\cdot R}(a\cdot 1_{M(R)})(b\cdot x),
	\end{eqnarray*}
for all $b\in A$. Similarly, $(b\cdot x)(a\cdot m)=(b\cdot x)(a\cdot 1_{M(R)})m|_{A\cdot R}$.
	
	\vd
	
	(ii) Let $a\in A$, $x\in R$ and $m\in R^{\underline{A}}$, then
	\begin{eqnarray*}
		m(a\cdot x)&\stackrel{\ref{def_acaoparcialmulti}(iii)}{=}&m(a\cdot (c\cdot x))\\
		&=&  m|_{A\cdot R}(a\cdot (c\cdot x))\\
		&\stackrel{\ref{pro_exteproacao}(i)}{=}& m|_{A\cdot R}(a_{(1)}\cdot 1_{M(R)})(a_{(2)}c\cdot x)\\
		&=&(a_{(1)} \cdot m)(a_{(2)}c\cdot x)\\
		&=&a\cdot (m(c\cdot x)).
	\end{eqnarray*}
\end{proof}

\begin{defi} Let $(R,\rho, E)$ be a symmetric partial $A$-comodule algebra. We define the \emph{coinvariant algebra} as the subalgebra of the elements of $M(R)$ which are invariant by $\rho$ as follows
	\begin{eqnarray}
	R^{\underline{coA}}=\{m\in M(R); \ \rho(m)=(m\otimes 1)E \ \mbox{and} \ \rho(m)=E(m\otimes 1)\}. \label{def_coinvparcial}
	\end{eqnarray}
\end{defi}

\begin{pro}\label{carac_coinv} If $(R, \rho, E)$ is a symmetric partial $A$-comodule algebra, then
	\begin{eqnarray*}
		R^{\underline{coA}}=\{ m\in M(R); \ w\cdot (mx)=m(w\cdot x) \ \mbox{and} \ w\cdot (xm)=(w\cdot x)m, x\in R \ \mbox{and} \ w\in \widehat{A} \}.
	\end{eqnarray*}
	\end{pro}
\begin{proof} Let $m\in R^{\underline{coA}}$, $w=\varphi(\underline{\hspace{0.3cm}} a)\in\widehat{A}$ and $x\in R$. Then
	\begin{eqnarray*}
		(w\cdot x)m &=& (\imath\otimes\varphi)(\rho(x)(1\otimes a))m\\
		&=& (\imath\otimes\varphi)(\rho(x)(1\otimes a)(m\otimes 1))\\
		&=&(\imath\otimes\varphi)(\rho(x)(m\otimes 1)(1\otimes a))\\
		&=& (\imath\otimes\varphi)(\rho(x)E(m\otimes 1)(1\otimes a))\\
		&=&  (\imath\otimes\varphi)(\rho(x)\rho(m)(1\otimes a))\\
		&=&  (\imath\otimes\varphi)(\rho(xm)(1\otimes a))\\
		&=& w\cdot(xm).
	\end{eqnarray*}
	 In a similar way, $w\cdot(mx)=m(w\cdot x)$.
	
	\vu
	
 Conversely, given $\varphi(\underline{\hspace{0.3cm}} c) \in \widehat{A}$ and writing $E(x\otimes a)=\sum_{k}\rho(y_k)(1\otimes b_k)$ we have that
	\begin{eqnarray*}
		(\imath\otimes\varphi(\underline{\hspace{0.3cm}} c))((m\otimes 1)E(x\otimes a))&=&m(\imath\otimes\varphi)(E(x\otimes a)(1\otimes c))\\
		&=& m(\imath\otimes\varphi)(\sum_{k}\rho(y_k)(1\otimes b_k)(1\otimes c))\\
		&=&\sum_{k}m((\imath\otimes\varphi)(\rho(y_k)(1\otimes b_kc)))\\
		&=& \sum_{k} m(\varphi(\underline{\hspace{0.3cm}} b_kc)\cdot y_k)\\
		&=& \sum_{k}(\varphi(\underline{\hspace{0.3cm}} b_kc)\cdot (my_k))\\
		&=& (\imath\otimes\varphi)\sum_{k}(\rho(my_k)(1\otimes b_kc))\\
		&=& (\imath\otimes\varphi(\underline{\hspace{0.3cm}} c))(\rho(m)E(x\otimes a))\\
		&=& (\imath\otimes\varphi(\underline{\hspace{0.3cm}} c))(\rho(m)(x\otimes a)),
	\end{eqnarray*}
	for all $c\in A$, then $(m\otimes 1)E(x\otimes a) = \rho(m)(x\otimes a)$, for all $x\in R$ and $a \in A$. Therefore, $(m\otimes 1)E=\rho(m)$. Similarly, $E(m\otimes 1)=\rho(m)$.
\end{proof}

\begin{defi} Let $(R, \rho, E)$ be a partial $A$-comodule algebra. We say that the partial coaction $\rho$ is \textbf{reduced} if it satisfies $(R\otimes 1)\rho(R)\subseteq (R\otimes A)E$.
\end{defi}

\begin{rem} Notice that in the case of reduced partial coactions, the inclusion above allows the use of the sigma notation because $(y\otimes 1)\rho(x)\in (R\otimes A)E$, for any $x,y\in R$, hence we can write $(y\otimes 1)\rho(x)=yx^{(0)}\otimes x^{(1)}$. Remember that in this notation one can not say that $x^{(0)} \otimes x^{(1)}$ belongs to the algebra $R\otimes A$, but all the term $yx^{(0)} \otimes x^{(1)}\in R\otimes A$. In the case that $R$ has local units, then one could think about the sigma notation for the coaction as in the case of the comultiplication.
\end{rem}

\begin{pro} If $(R,\rho, E)$ is a reduced partial coaction, then $\rho(R)(R\otimes 1)\subseteq E(R\otimes A)$.
\end{pro}
\begin{proof}
	Indeed, let $x,y\in R$, $a,b\in A$ and write $(1\otimes a)\rho(x)=\sum_i x_iy\otimes a_i$, then
		\begin{eqnarray*}
		&& \hspace{-2cm}(1\otimes a)\rho(x)(y\otimes 1)(1\otimes S(b))=\\
		&=& \sum_i x_iy\otimes a_iS(b)\\
		&=&  \sum_i x_iy^{(0)} \otimes a_iS(b_{(1)})\varepsilon(b_{(2)}y^{(1)})\\
		&=& \sum_i x_iy^{(0)} \otimes m(\imath\otimes S)((a_iS(b_{(1)})\otimes 1)\Delta(b_{(2)}y^{(1)}))\\
		&=& (\imath\otimes m(\imath\otimes S))(\sum_i(x_i\otimes a_iS(b_{(1)})\otimes 1)(\imath\otimes \Delta)((1\otimes b_{(2)})\rho(y)))\\
		&=& (\imath\otimes m(\imath\otimes S))((\sum_ix_i\otimes (a_iS(b_{(1)})\otimes 1)\Delta(b_{(2)}))(\imath\otimes\Delta)(\rho(y)))\\
		&=& (\imath\otimes m(\imath\otimes S))((\sum_ix_i\otimes a_i\otimes b)(\imath\otimes\Delta)(\rho(y)))\\
		&=& (\imath\otimes m(\imath\otimes S))(((1\otimes a)\rho(x)\otimes b)(\imath\otimes\Delta)(\rho(y)))\\
		&\stackrel{\ref{rhoigual_rhoe}}{=}& (\imath\otimes m(\imath\otimes S))(((1\otimes a)\rho(x)E\otimes b)(\imath\otimes\Delta)(\rho(y)))\\
		&=& (\imath\otimes m(\imath\otimes S))(((1\otimes a)\rho(x)\otimes 1)(E\otimes b)(\imath\otimes\Delta)(\rho(y)))\\
		&\stackrel{(\ref{escritacoa2})}{=}& (\imath\otimes m(\imath\otimes S))((1\otimes a\otimes 1)(\rho(x)\otimes 1)(\rho\otimes \imath)((1\otimes b)\rho(y)))\\
		&=& (\imath\otimes m(\imath\otimes S))((1\otimes a)\rho(xy^{(0}))\otimes by^{(1)})\\
		&=&(1\otimes a)\rho(xy^{(0)})(1\otimes S(y^{(1)}))(1\otimes S(b)),
	\end{eqnarray*}
	for all $a,b\in A$, thus  $\rho(x)(y\otimes 1)=\rho(xy^{(0)})(1\otimes S(y^{(1)}))\in E(R\otimes A)$, by definition of a partial coaction.
\end{proof}

\begin{rem} If $(R, \rho, E)$ is a reduced symmetric partial $A$-comodule algebra, then we can define the following linear map
	\begin{eqnarray*}
		\beta: R\otimes_{R^{\underline{coA}}} R &\longrightarrow & (R\otimes A)E\\
		x\otimes y &\longmapsto & (x\otimes 1)\rho(y).
	\end{eqnarray*}
\end{rem}

\begin{exa} If $R$ is a reduced $A$-comodule algebra and $f$ is a central idempotent in $R$, then $L=fR$ is a reduced symmetric partial $A$-comodule algebra, by Proposition \ref{pro_coainduzido}.
\end{exa}

\begin{exa} Consider Proposition \ref{coalambda}. If $m\in A$, then $(R, \rho, E)$ is a reduced partial $A$-comodule algebra. 
\end{exa}

The following results will be useful in the next section.

\begin{lem} Let $(R,\rho,E)$ be a reduced symmetric partial $A$-comodule algebra. Then
	\begin{eqnarray*}
		(\rho\otimes \imath)(\rho(x)(y\otimes 1))(1\otimes a\otimes 1)=\sum_j(\imath\otimes\Delta)(\rho(x)(z_j\otimes 1))(1\otimes b_j\otimes 1),
	\end{eqnarray*}
	where $\rho(y)(1\otimes a)=\sum_jz_j\otimes b_j$, for all $x,y\in R$ and $a\in A$.
	\label{escri_estr1}
\end{lem}
\begin{proof}
	Let $x,y\in R$ and $a\in A$. Then
	\begin{eqnarray*}
		(\rho\otimes \imath)(\rho(x)(y\otimes 1))(1\otimes a\otimes c)&=& (\rho\otimes \imath)(\rho(x)(1\otimes c)(y\otimes 1))(1\otimes a\otimes 1)\\
		&=& (\rho\otimes \imath)(\rho(x)(1\otimes c))(\rho(y)(1\otimes a)\otimes 1)\\
		&\stackrel{(\ref{escritacoa3})}{=}&(\imath\otimes\Delta)(\rho(x))(E\otimes 1)(1\otimes 1\otimes c)(\rho(y)(1\otimes a)\otimes 1)\\
		&=& (\imath\otimes\Delta)(\rho(x))(\rho(y)(1\otimes a)\otimes c)\\
		&=& (\imath\otimes\Delta)(\rho(x))(\sum_jz_j\otimes b_j\otimes c)\\
		&=&\sum_j(\imath\otimes\Delta)(\rho(x)(z_j\otimes 1))(1\otimes b_j\otimes c).
	\end{eqnarray*}
\end{proof}

\begin{lem} If $(R,\rho,E)$ is a reduced symmetric partial $A$-comodule algebra then
	\begin{center}
		$(\rho\otimes\imath)((x\otimes 1)\rho(y))=(\rho(x)\otimes 1)(\imath\otimes\Delta)(\rho(y)),$
	\end{center}
	for all $x,y\in R$.
	\label{reduzi}
\end{lem}
\begin{proof}
	Let $x,y\in R$ and $a\in A$,
	\begin{eqnarray*}
		(\rho\otimes\imath)((x\otimes 1)\rho(y))(1\otimes 1\otimes a)&=&(\rho\otimes \imath)((x\otimes 1)\rho(y)(1\otimes a))\\
		&=&(\rho(x)\otimes 1)(E\otimes 1)(\imath\otimes\Delta)(\rho(y))(1\otimes 1\otimes a)\\
		&=&(\rho(x)\otimes 1)(\imath\otimes\Delta)(\rho(y))(1\otimes 1\otimes a).
	\end{eqnarray*}
\end{proof}

\subsection{A Morita Context}
Our aim in this section is to construct a Morita context relating the coinvariant algebra $R^{\underline{coA}}$ and a subalgebra of the smash product $R\# A$, in the following situation:
$A$ is a regular multiplier Hopf algebra with integrals, $\widehat{A}$ is its dual algebra, $R$ is an algebra with a nondegenerate product such that $R^2=R$, and  $(R, \rho, E)$ a reduced symmetric partial $A$-comodule.

\vu
 
In \cite{Frame} it is shown the existence of a unique invertible element $\delta\in M(A)$ such that
\begin{equation}\label{grouplike}
	(\varphi\otimes\imath)\Delta(a)=\varphi(a)\delta,
\end{equation}
for all $a\in A$, whose the inverse is given by $S(\delta)$, and $\Delta(\delta)=\delta\otimes\delta$. We denote $\widehat{a}=\varphi(\underline{\hspace{0.3cm}} a)$ and  $\widehat{a}^{\delta}=\varphi(\underline{\hspace{0.3cm}} \delta a)$.

\begin{defi}
	Let $(R, \rho, E)$ be a partial $A$-comodule algebra and $$\Omega = \{(\imath \otimes \varphi)(\rho(x)(1 \otimes a))\ ;\ x\in R, a\in A \}\subseteq R.$$ We say that $\rho$ is a \emph{restrict} partial coaction if it is reduced and there is $a\in A$ such that
	$(\imath \otimes \varphi)(E(1 \otimes a)) = 1_{M(R)}|_{\Omega}.$
\end{defi}

Observe that if $R$ is a restrict symmetric partial $A$-comodule algebra, then  $R$ is a symmetric partial $\widehat{A}$-module algebra and the restriction assumption implies that the product in $\widehat{A} \cdot R$ is nondegenerate. Hence, the symmetric partial action of $\widehat{A}$ on $R$ can be extended to a linear map of $\widehat{A}$ on $M(R)$. The extension of the partial action of $\widehat{A}$ on $M(R)$ is fundamental for the construction of the algebras that will appear in the sequel.

\begin{exa}
	Consider the partial coaction given by Example \ref{ex_concretstrito} and take $a=\delta_q\in A_G$, with $q\in N$. Then $L$ is a restrict symmetric partial $A_G$-comodule algebra.
\end{exa}

\begin{exa} Consider the partial coaction given by Example \ref{ex_mlambda} and take $a=\delta_q\in A_G$, with $q\in N$. Then $R$ is a restrict symmetric partial $A_G$-comodule algebra.
\end{exa}

\begin{pro} If $(R,\rho,E)$ is a restrict symmetric partial $A$-comodule algebra, then
	\begin{enumerate}
		\item[(i)] $(\widehat{A}\cdot R)\# \widehat{A}$ is a subalgebra of $R\# \widehat{A}$;
		\item[(ii)] $R^{\underline{coA}}\subseteq R^{\underline{\widehat{A}}}$.
	\end{enumerate}
	\label{pro_algcontMorit}
\end{pro}
\begin{proof} 
	
	(i) By Proposition \ref{pro_igualextensaoacao} it follows that $\e(\widehat{A})R=\widehat{A}\cdot R=R\e(\widehat{A})$ is a subalgebra of $R$, thus $\widehat{A}\cdot R$ is a partial $\widehat{A}$-submodule algebra of $R$ because
	\begin{center}
		$\varphi(\underline{\hspace{0.3cm}}a)\cdot (\varphi(\underline{\hspace{0.3cm}} b)\cdot x)\stackrel{\ref{pro_exteproacao}}{=}\e(\varphi(\underline{\hspace{0.3cm}}a)_{(1)})(\varphi(\underline{\hspace{0.3cm}} a)_{(2)}\varphi(\underline{\hspace{0.3cm}} b)\cdot x)\in \e(\widehat{A})R=\widehat{A}\cdot R$,
	\end{center}
	for all $\varphi(\underline{\hspace{0.3cm}}a), \varphi(\underline{\hspace{0.3cm}}b)\in\widehat{A}$ and $x\in R$. Therefore, $(\widehat{A}\cdot R)\# \widehat{A}$ is a subalgebra of $R\# \widehat{A}$.
	
	(ii) This result follows from Propositions \ref{carac_invari} and \ref{carac_coinv}.
\end{proof}

It follows bellow  some useful results for the construction of the Morita context.

\begin{lem}\label{bimodulestructure} Under the above conditions, $\widehat{A}\cdot R$ is a unitary $((\widehat{A}\cdot R)\# \widehat{A}, R^{\underline{coA}})$-bimodule  with the following structure
	\begin{eqnarray*}
		(x\# \widehat{a})\triangleright y&=& x(\widehat{a}\cdot y)\\
		x\triangleleft m&=&xm,
	\end{eqnarray*}
	for all $x,y \in \widehat{A}\cdot R$, $\widehat{a}\in \widehat{A}$ and $m\in R^{\underline{coA}}$.
	\label{bimodu1}
\end{lem}

\begin{proof}
	We begin verifying that $\widehat{A}\cdot R$ is a left unitary $((\widehat{A}\cdot R)\#\widehat{A})$-module. In fact, let $x,y,z\in\widehat{A}\cdot R$ and $\widehat{a},\widehat{b}\in \widehat{A}$,
	\begin{eqnarray*}
		(x\#\widehat{a})\triangleright((y\#\widehat{b})\triangleright z)
		&=& x(\widehat{a}\cdot (y(\widehat{b}\cdot z)))\\
		&=& x(\widehat{a}_{(1)}\cdot y)(\widehat{a}_{(2)}\widehat{b}\cdot z)\\
		&=& (x(\widehat{a}_{(1)}\cdot y)\# \widehat{a}_{(2)}\widehat{b})\triangleright z\\
		&=& ((x\#\widehat{a})(y\#\widehat{b}))\triangleright z.
	\end{eqnarray*}
	
	And, the fact that the module is unitary it follows from  $R^2=R$ and $\e(\widehat{A})R=\widehat{A}\cdot R=R\e(\widehat{A})$. Furthermore, it follows directly from Proposition \ref{carac_coinv} and the fact that $1_{M(R)}\in R^{\underline{coA}}$, that $\widehat{A}\cdot R$ is a right unitary $R^{\underline{coA}}$-module and  
	\begin{eqnarray*}	
((x \# a) \triangleright z ) \triangleleft m &=& (x(\widehat{a} \cdot z))m\\	
&=& (x(\widehat{a} \cdot z)m)\\	
&\stackrel{\ref{carac_coinv}}{=}& x(\widehat{a} \cdot zm)\\
&=&(x \# \widehat{a}) \triangleright (zm)\\
&=&(x \# \widehat{a}) \triangleright (z \triangleleft m).\\
\end{eqnarray*}	
Hence, $\widehat{A}\cdot R$ is a $((\widehat{A}\cdot R)\# \widehat{A}, R^{\underline{coA}})$-bimodule.

\end{proof}

\begin{lem} Under the above conditions, $\widehat{A}\cdot R$ is a unitary  $(R^{\underline{coA}},(\widehat{A}\cdot R)\# \widehat{A})$-bimodule with the following structure
	\begin{eqnarray*}
		x\triangleleft(y\#\widehat{a})&=& S^{-1}(\widehat{a}^{\delta})\cdot(xy)\\
		m\triangleright x&=& mx,
	\end{eqnarray*}
	for all $x,y \in \widehat{A}\cdot R$, $\widehat{a} \in \widehat{A}$ and $m\in R^{\underline{coA}}$.
	\label{bimodu2}
\end{lem}
\begin{proof}
	It is enough to check that $\widehat{A}\cdot R$ is a right $((\widehat{A}\cdot R)\# \widehat{A})$-module,  the other statements follow in a similar way from Lemma \ref{bimodulestructure}. Indeed, let $x\in\widehat{A}\cdot R$ and $y\#\widehat{a},z\#\widehat{b} \in (\widehat{A}\cdot R)\#\widehat{A}$,
	\begin{eqnarray*}
		(x\triangleleft (y\#\widehat{a}))\triangleleft (z\#\widehat{b})&=& S^{-1}(\widehat{b}^{\delta})\cdot ((S^{-1}(\widehat{a}^{\delta})\cdot (xy))z)\\
		&\stackrel{\ref{def_acaoparcialmulti}(v)}{=}&(S^{-1}((\widehat{b}^{\delta})_{(2)})S^{-1}(\widehat{a}^{\delta})\cdot (xy))
(S^{-1}((\widehat{b}^{\delta})_{(1)})\cdot z)\\
		&=& (S^{-1}(\widehat{a}^{\delta}(\widehat{b}^{\delta})_{(2)})\cdot (xy))(S^{-1}((\widehat{b}^{\delta})_{(1)})\cdot z)\\
		&=& (S^{-1}((\widehat{a}^{\delta})_{(3)}(\widehat{b}^{\delta})_{(2)})\cdot (xy))(S^{-1}((\widehat{b}^{\delta})_{(1)})
S^{-1}((\widehat{a}^{\delta})_{(2)})(\widehat{a}^{\delta})_{(1)}\cdot z)\\
		&=& (S^{-1}(((\widehat{a}^{\delta})_{(2)} \widehat{b}^{\delta})_{(2)})\cdot (xy))(S^{-1}(((\widehat{a}^{\delta})_{(2)} 
\widehat{b}^{\delta})_{(1)})(\widehat{a}^{\delta})_{(1)}\cdot z)\\
		&\stackrel{\ref{def_acaoparcialmulti}(i)}{=}& S^{-1}((\widehat{a}^{\delta})_{(2)} \widehat{b}^{\delta})\cdot 
((xy)((\widehat{a}^{\delta})_{(1)}\cdot z))\\
		&=& S^{-1}((\widehat{a}_{(2)} \widehat{b})^{\delta}) \cdot (x(y(\widehat{a}_{(1)}\cdot z)))\\
		&=& x\triangleleft (y(\widehat{a}_{(1)}\cdot z)\# \widehat{a}_{(2)} \widehat{b})\\
		&=& x\triangleleft( (y\#\widehat{a})(z\#\widehat{b})).
	\end{eqnarray*}
\end{proof}

\begin{pro} Let $(R,\rho,E)$ be a restrict symmetric partial coaction. Then the linear map		
	\begin{eqnarray*}
		( \ , \ ): (\widehat{A}\cdot R)\otimes_{ (\widehat{A}\cdot R)\# \widehat{A}} (\widehat{A}\cdot R) & \longrightarrow & {R^{\underline{coA}}}\\
		x\otimes\ y & \longmapsto & (x,y)=(\imath\otimes\varphi)\rho(xy)
	\end{eqnarray*}
	is  ${R^{\underline{coA}}}$-bilinear and satisfies $(x\triangleleft(y\# \widehat{a}),z)=(x,(y\# \widehat{a})\triangleright z)$,
	for all $x, y, z\in \widehat{A}\cdot R$, $ \widehat{a}\in \widehat{A}.$
	\label{apli_parent}
\end{pro}
\begin{proof} For every $x\in\widehat{A}\cdot R$, define $(id\otimes\varphi)\rho(x)\in M(R)$, as follows,
	\begin{eqnarray*}
		((\imath\otimes\varphi)\rho(x))y&=& (\imath\otimes\varphi)(\rho(x)(y\otimes 1))\\
		y((\imath\otimes\varphi)\rho(x))&=& (\imath\otimes\varphi)((y\otimes 1)\rho(x)),
	\end{eqnarray*}
	for all $y\in R$, thus  $(\imath\otimes\varphi)\rho(x)\in R^{\underline{coA}}$.
	
	Indeed, let $s\otimes a\in R\otimes A$ and write $E(s\otimes a)=\sum\limits_i\rho(s_i)(1\otimes a_i)=\sum\limits_{i,j}r_{ij}\otimes b_{ij}$, 
	\begin{eqnarray*}
		\rho((\imath\otimes\varphi)\rho(x))(s\otimes a)&=& \rho((\imath\otimes\varphi)\rho(x))(\sum\limits_i\rho(s_i)(1\otimes a_i))\\
		&=& \sum_i\rho((\imath\otimes\varphi)(\rho(x)(s_i\otimes 1)))(1\otimes a_i)\\
		&=& \sum_i\rho(x^{(0)}s_i)(1\otimes a_i)\varphi(x^{(1)})\\
		&=& (\imath\otimes\imath\otimes\varphi)(\sum_i(\rho\otimes\imath)(\rho(x)(s_i\otimes 1))(1\otimes a_i\otimes 1))\\
		&\stackrel{\ref{escri_estr1}}{=}& (\imath\otimes\imath\otimes\varphi)(\sum_{i,j}(\imath\otimes\Delta)(\rho(x)(r_{ij}\otimes 1))(1\otimes b_{ij}\otimes 1))\\
		&=& \sum_{i,j} x^{(0)}r_{ij}\otimes b_{ij}\varphi(x^{(1)})\\
		&=& \sum_{i,j} (\imath\otimes\varphi)(\rho(x)(r_{ij}\otimes 1))\otimes b_{ij}\\
		&=& \sum_{i,j}(((\imath\otimes\varphi)\rho(x))r_{ij})\otimes b_{ij}\\
		&=& ((\imath\otimes\varphi)\rho(x)\otimes 1)(\sum_i\rho(s_i)(1\otimes a_i))\\
		&=&((\imath\otimes\varphi)\rho(x)\otimes 1)E(s\otimes a),
	\end{eqnarray*}
	for all $s\otimes a\in R\otimes A$, concluding that the linear map $( \ , \ )$ is well defined.
	
	\vu
	
	To verify that the map $( \ , \ )$ is $((\widehat{A}\cdot R)\#\widehat{A})$-balanced, consider $x,y,z\in\widehat{A}\cdot R$, $\widehat{a}\in\widehat{A}$ and write $\widehat{a}^{\delta}=\varphi(\underline{\hspace{0.3cm}} \delta a)=\varphi(c\ \underline{\hspace{0.3cm}} )$, thus $S^{-1}(\widehat{a}^{\delta})=\varphi\circ S^{-1}(\underline{\hspace{0.3cm}} \ S(c))$ and,
	\begin{eqnarray*}
		& &\hspace{-2.5cm}(y\triangleleft(x\#\widehat{a}),z)(r)=\\
		\hspace{1cm}&=&(S^{-1}(\widehat{a}^{\delta})\cdot(yx),z)(r)\\
		&=& ((yx)^{(0)},z)(r)\varphi\circ S^{-1}((yx)^{(1)}S(c))\\
		&=& ((\imath\otimes\varphi)\rho((xy)^{(0)}z))r\varphi\circ S^{-1}((yx)^{(1)}S(c))\\
		&=& (\imath\otimes\varphi)(\rho((xy)^{(0)}z)(r\otimes 1))\varphi\circ S^{-1}((yx)^{(1)}S(c))\\
		&=& (\imath\otimes\varphi\otimes \varphi\circ S^{-1})((\rho\otimes \imath)(\rho(yx)(1\otimes S(c))(z\otimes 1))(r\otimes 1\otimes 1))\\
		&\stackrel{(\ref{escritacoa3})}{=}& (\imath\otimes\varphi\otimes \varphi\circ S^{-1})((\imath\otimes\Delta)(\rho(yx))(E\otimes 1)(\rho(z)(r\otimes 1)\otimes S(c)))\\
		&=& (\imath\otimes\varphi\otimes \varphi\circ S^{-1})((\imath\otimes\Delta)(\rho(yx))(\rho(z)(r\otimes 1)\otimes S(c)))\\
		&=& (\imath\otimes\varphi\otimes \varphi\circ S^{-1})((\imath\otimes\Delta)(\rho(yx))(\sum_ir_i\otimes d_i\otimes S(c)))\\
		&=& (\imath\otimes\varphi\otimes \varphi\circ S^{-1})(\sum_i(yx)^{(0)}r_i\otimes\Delta((yx)^{(1)})(d_i\otimes S(c)))\\
		&=& \sum_i  (yx)^{(0)}r_i\varphi(((yx)^{(1)})_{(1)}d_i)\varphi(cS^{-1}(((yx)^{(1)})_{(2)}))\\
		&=& \sum_i  (yx)^{(0)}r_i\varphi(((yx)^{(1)})_{(1)}d_i)\varphi(S^{-1}(((yx)^{(1)})_{(2)})\delta a)\\
		&\stackrel{(\ast)}{=}& \sum_i  (yx)^{(0)}r_i\varphi (S^{-1}(((yx)^{(1)})_{(2)}) ((\varphi\otimes \imath )\Delta (((yx)^{(1)})_{(1)}d_i)) a)\\
		&=& \sum_i  (yx)^{(0)}r_i\varphi (S^{-1}(((yx)^{(1)})_{(2)})(\varphi\otimes \imath )(\Delta (((yx)^{(1)})_{(1)}d_i)(1\otimes a)))\\
		&=& \sum_i  (yx)^{(0)}r_i\varphi(((yx)^{(1)})_{(1)}(d_i)_{(1)}) \varphi (S^{-1}(((yx)^{(1)})_{(3)})((yx)^{(1)})_{(2)}(d_i)_{(2)}a)\\
		&=& \sum_i  (yx)^{(0)}r_i\varphi ((yx)^{(1)}(d_i)_{(1)})\varphi ((d_i)_{(2)} a)\\
		&=& (\imath\otimes\varphi\otimes\varphi)(\sum_i (\rho(yx)(r_i\otimes 1)\otimes 1)(1\otimes\Delta(d_i)(1\otimes a))\\
		&=& (\imath\otimes\varphi\otimes\varphi)((\rho(yx)\otimes 1)((\imath\otimes\Delta)(\sum_ir_i\otimes d_i)(1\otimes 1\otimes a)))\\
		&=& (\imath\otimes\varphi\otimes\varphi)((\rho(yx)\otimes 1)((\imath\otimes\Delta)(\rho(z)(r\otimes 1))(1\otimes 1\otimes a))),
	\end{eqnarray*}
where in $(\ast)$ we used $(\varphi\otimes \imath)\Delta(a)=\varphi(a)\delta$. On the other hand,
		\begin{eqnarray*}
		&& \hspace{-2cm}(y, (x\#\widehat{a})\triangleright z)(r)=\\
		&=& ((\imath\otimes\varphi)(\rho(yx(\widehat{a}\cdot z))))(r)\\
		&=& (\imath\otimes\varphi)(\rho(yx(\widehat{a}\cdot z))(r\otimes 1))\\
		&=& (\imath\otimes\varphi)(\rho(yxz^{(0)})(r\otimes 1))\varphi(z^{(1)}a)\\
		&=& (\imath\otimes\varphi\otimes\varphi)((\rho(yx)\otimes 1)(\rho\otimes\imath)(\rho(z)(1\otimes a))(r\otimes 1\otimes 1))\\
		&\stackrel{(\ref{escritacoa1})}{=}&(\imath\otimes\varphi\otimes\varphi)((\rho(yx)\otimes 1)(E\otimes 1)(\imath\otimes\Delta)(\rho(z))(r\otimes 1\otimes a))\\
		&\stackrel{\ref{rhoigual_rhoe}}{=}& (\imath\otimes\varphi\otimes\varphi)((\rho(yx)\otimes 1)((\imath\otimes\Delta)(\rho(z)(r\otimes 1))(1\otimes 1\otimes a))),
	\end{eqnarray*}
	for any $r\in R$, therefore $(y\triangleleft(x\#\widehat{a}),z)=(y, (x\#\widehat{a})\triangleright z)$.
	
	\vu
	
	The bilinearity follows in a natural way because for $x,y\in\widehat{A}\cdot R$ and $m\in R^{\underline{coA}}$, then, we have
	\begin{eqnarray*}
		(m\triangleright x,y)(r)
		&=& (\imath\otimes\varphi)(\rho(m(xy))(r\otimes 1))\\
		&\stackrel{(\ref{def_coinvparcial})}{=}& (\imath\otimes\varphi)((m\otimes 1)E\rho(xy)(r\otimes 1))\\
		&=& m((xy)^{(0)}r)\varphi((xy)^{(1)})\\
		&=& (m(x,y))(r),
	\end{eqnarray*}
	for all $r\in R$.
\end{proof}

\begin{lem}
	Let $(R,\rho,E)$ be a restrict symmetric partial coaction. Then exists a well defined linear map		
	\begin{eqnarray*}
		\theta:(\widehat{A}\cdot R) \otimes_{R^{\underline{coA}}} (\widehat{A}\cdot R)& \longrightarrow & (\widehat{A}\cdot R) \# \widehat{A}\\
		x \otimes y &\longmapsto & \theta(x \otimes y):= xy^{(0)}\# \varphi(y^{(1)}\underline{\hspace{0.3cm}})
	\end{eqnarray*}
satisfying $\theta(x\otimes y)(z\otimes a)= \theta(x\otimes y)(E(z\otimes a))$, for $x,y,z\in \widehat{A}\cdot R$ and $a\in A$.
	\label{aplicteta}
\end{lem}
\begin{proof}
	Let $x,y\in \widehat{A}\cdot R$ and $m\in R^{\underline{coA}}$. Then
	\begin{eqnarray*}
		\theta(x\triangleleft m, y)	&=& (xm)y^{(0)}\# \varphi(y^{(1)} \underline{\hspace{0.3cm}})\\
		&=& (\imath\otimes\varphi)((x\otimes 1)(m\otimes 1)E\rho(y)(1\otimes \underline{\hspace{0.3cm}}))\\
		&\stackrel{(\ref{def_coinvparcial})}{=}& (\imath\otimes\varphi)((x\otimes 1)\rho(m)\rho(y)(1\otimes \underline{\hspace{0.3cm}}))\\
		&=& x(my)^{(0)}\# \varphi((my)^{(1)}\underline{\hspace{0.3cm}} )\\
		&=& \theta(x, m\triangleright y),
	\end{eqnarray*}
	which means that is $R^{\underline{coA}}$-balanced. And, given $x,y,z\in \widehat{A}\cdot R$, where $y=\sum_i\widehat{b_i}\cdot y_i$ and $a\in A$, we obtain
	\begin{eqnarray*}
		\theta(x\otimes y)(z\otimes a)&=& \theta(x\otimes \sum_i\widehat{b_i}\cdot y_i)(z\otimes a)\\
		&=& \sum_i \theta(x\otimes (y_i)^{(0)})\varphi((y_i)^{(1)}b_i)(z\otimes a)\\
		&=&\sum_i(x(y_i)^{(0)(0)}\#\varphi((y_i)^{(0)(1)}\underline{\hspace{0.3cm}} )\varphi((y_i)^{(1)}b_i))(z\otimes a)\\
		&=& (\imath\otimes\varphi\otimes\varphi)(\sum_i x(y_i)^{(0)(0)}z\otimes (y_i)^{(0)(1)}a\otimes (y_i)^{(1)}b_i)\\
		&=&  (\imath\otimes\varphi\otimes\varphi)(\sum_i(x\otimes 1\otimes 1)(\rho\otimes \imath)(\rho(y_i)(1\otimes b_i))(z\otimes a\otimes 1))\\
		&\stackrel{(\ref{escritacoa3})}{=}&  (\imath\otimes\varphi\otimes\varphi)(\sum_i(x\otimes 1\otimes 1)(\imath\otimes\Delta)(\rho(y_i))(E(z\otimes a)\otimes b_i)).\\
	\end{eqnarray*}
	
	On the other hand, repeating the above process, we have
	
	\begin{eqnarray*}
		\theta(x\otimes y)(E(z\otimes a))&=& \theta(x\otimes \sum_i\widehat{b_i}\cdot y_i)(\sum_j z_j\otimes c_j)\\
		&=& (\imath\otimes\varphi\otimes\varphi)(\sum_i(x\otimes 1\otimes 1)(\imath\otimes\Delta)(\rho(y_i))(E(\sum_j z_j\otimes c_j)\otimes b_i))\\
		&=& (\imath\otimes\varphi\otimes\varphi)(\sum_i(x\otimes 1\otimes 1)(\imath\otimes\Delta)(\rho(y_i))(E(E(z\otimes a))\otimes b_i))\\
		&=&  (\imath\otimes\varphi\otimes\varphi)(\sum_i(x\otimes 1\otimes 1)(\imath\otimes\Delta)(\rho(y_i))(E(z\otimes a)\otimes b_i)).
	\end{eqnarray*}
	
	Hence, $\theta(x\otimes y)(z\otimes a)=\theta(x\otimes y)(E(z\otimes a))$.
\end{proof}

\begin{rem} Under these conditions, Lemma \ref{aplicteta} suggests  to define the following algebra
	$$B:= ((\widehat{A}\cdot R) \# \widehat{A})|_{E((\widehat{A}\cdot R) \otimes A)}$$
	with product given by
	$$(x\# \widehat{a})|_{E((\widehat{A}\cdot R) \otimes A)} (y\# \widehat{b})|_{E((\widehat{A}\cdot R) \otimes A)} = [(x\# \widehat{a})(y\# \widehat{b})]|_{E((\widehat{A}\cdot R) \otimes A)},$$
	for all $x,y\in\widehat{A}\cdot R$ and $\widehat{a},\widehat{b}\in\widehat{A}$.

\vu

	Therefore, it is important to observe that the results \ref{bimodu1}, \ref{bimodu2} e \ref{apli_parent} still remains true for the algebra $B$. Then in what follows, we will write this algebra $B$ just as $(\widehat{A}\cdot R)\# \widehat{A}$, in order to do not overload the notation. 
\end{rem}

\begin{pro}\label{aplic_colche} Let $(R,\rho,E)$ be a restrict symmetric partial coaction. Then the linear map
	\begin{eqnarray*}
		[ \ , \ ]:(\widehat{A}\cdot R)\otimes_{R^{\underline{coA}}}(\widehat{A}\cdot R) & \longrightarrow &(\widehat{A}\cdot R)\#\widehat{A}\\
		x\otimes y & \longmapsto & [x,y]=xy^{(0)}\#\varphi(y^{(1)}\underline{\hspace{0.3cm}})
	\end{eqnarray*}
	is $((\widehat{A}\cdot R)\#\widehat{A})$-bilinear and $[x\triangleleft m,y]=[x,m\triangleright y],$ for all $x, y\in\widehat{A}\cdot R$, $m\in {R^{\underline{coA}}}$.
	\end{pro}
\begin{proof}
	By Lemma \ref{aplicteta} it only remains to check the  bilinearity of the map $[ \ , \ ]$. Indeed, consider $\widehat{a}\in\widehat{A}$ and $x,y,z\in\widehat{A}\cdot R$,
	\begin{eqnarray*}
		& &\hspace{-2cm}((x\#\widehat{a})[y,z])(w\otimes b)=\\
		&=& ((x\# \widehat{a})(yz^{(0)}\#\varphi(z^{(1)}\underline{\hspace{0.3cm}} )))(w\otimes b)\\
		&=& (x(\widehat{a}_{(1)}\cdot yz^{(0)})\# \widehat{a}_{(2)}\varphi(z^{(1)}\underline{\hspace{0.3cm}} ))(w\otimes b)\\
		&\stackrel{(*)}{=}& \sum_i x(yz^{(0)})^{(0)}w\varphi((yz^{(0)})^{(1)}c_i)\varphi(b d_i)\\
		&=& \sum_i (\imath\otimes\varphi)((x\otimes 1)\rho(yz^{(0)})(w\otimes 1)(1\otimes c_i))\varphi(b d_i)\\
		&=& \sum_i (\imath\otimes\varphi)((x\otimes 1)((yz^{(0)})^{(0)}w\otimes(yz^{(0)})^{(1)})(1\otimes c_i))\varphi(b d_i)\\
		&=& x ((yz^{(0)})^{(0)}w)(\sum_i\varphi(\underline{\hspace{0.3cm}} c_i)\otimes
\varphi(\underline{\hspace{0.3cm}} d_i))((yz^{(0)})^{(1)}\otimes b)\\
		&\stackrel{(**)}{=}& x ((yz^{(0)})^{(0)}w)(\widehat{\Delta}(\widehat{a})(1\otimes\varphi (z^{(1)}\underline{\hspace{0.3cm}} )))
((yz^{(0)})^{(1)}\otimes b)\\
		&=& x ((yz^{(0)})^{(0)}w)(\widehat{a}\otimes \varphi(z^{(1)}\underline{\hspace{0.3cm}} ))(((yz^{(0)})^{(1)}\otimes 1)\Delta(b))\\
		&=& x ((yz^{(0)})^{(0)}w)(\varphi\otimes\varphi(z^{(1)}\underline{\hspace{0.3cm}} ))(((yz^{(0)})^{(1)}\otimes 1)\Delta(b)(a\otimes 1))\\
		&=&  x ((yz^{(0)})^{(0)}w)\varphi((yz^{(0)})^{(1)}(b_{(1)}a))\varphi(z^{(1)}b_{(2)})\\
		&=& (\imath\otimes\varphi\otimes\varphi)((x\otimes 1\otimes 1)(\rho\otimes \imath)((y\otimes 1)\rho(z))(w\otimes \Delta(b)(a\otimes 1))),
	\end{eqnarray*}
		where, for the equalities $(*)$ and $(**)$ we use  $\widehat{\Delta}(\widehat{a})
(1\otimes \varphi(z^{(1)}\underline{\hspace{0.3cm}} ))=\sum_i\varphi(\underline{\hspace{0.3cm}} c_i)\otimes\varphi(\underline{\hspace{0.3cm}} d_i)$. 

On the other hand, for all $w\otimes b\in E((\widehat{A}\cdot R)\otimes A)$,
		\begin{eqnarray*}
		&&\hspace{-1.5cm} [(x\#\widehat{a})\triangleright y,z](w\otimes b)=\\
		&=& [x(\widehat{a}\cdot y),z](w\otimes b)\\		
		&=& (x(\widehat{a}\cdot y)z^{(0)}\otimes\varphi(z^{(1)}\underline{\hspace{0.3cm}} ))(w\otimes b)\\
		&=& x(\widehat{a}\cdot y)z^{(0)}w\varphi(z^{(1)}b)\\
		&=& xy^{(0)}\varphi(y^{(1)}a)z^{(0)}w\varphi(z^{(1)}b)\\
		&=& (xy^{(0)})z^{(0)}w\varphi(y^{(1)}a)\varphi(z^{(1)}b)\\
		&\stackrel{(\ast)}{=}& (\imath\otimes\varphi\otimes\varphi)((x\otimes 1)\rho(y)\otimes 1)(\imath\otimes\Delta)(\rho(z)(1\otimes b))(w\otimes a\otimes 1))\\
		&=& (\imath\otimes\varphi\otimes\varphi)((x\otimes 1\otimes 1)(\rho(y)\otimes 1)(\imath\otimes\Delta)(\rho(z))(w\otimes \Delta(b)(a\otimes 1)))\\
		&\stackrel{\ref{reduzi}}{=}& (\imath\otimes\varphi\otimes\varphi)((x\otimes 1\otimes 1)(\rho\otimes \imath)((y\otimes 1)\rho(z))(w\otimes\Delta(b)(a\otimes 1))),	
	\end{eqnarray*}
 in which in $(*)$ we used the left invariance of the integral $\varphi$. Thus, $(x\#\widehat{a})[y,z]=[(x\#\widehat{a})\triangleright y,z]$.
	
	\vu
	
	Also, we have,
	\begin{eqnarray*}
		([y,z](x\#\widehat{a}))(w\otimes b)&=&((yz^{(0)})(\varphi(z^{(1)}\underline{\hspace{0.3cm}} )_{(1)}\cdot x)\# 
\varphi(z^1\underline{\hspace{0.3cm}} )_{(2)}\widehat{a})(w\otimes b)\\
		&=&(\sum_i (yz^{(0)})(\varphi(\underline{\hspace{0.3cm}} c_i)\cdot x) \# \varphi(\underline{\hspace{0.3cm}} d_i))(w\otimes b)\\
		&=& \sum_i (\imath\otimes\varphi)((yz^{(0)}\otimes 1)\rho(x)(1\otimes c_i))w\varphi(bd_i)\\
		&=&  \sum_i (\imath\otimes\varphi)(((yz^{(0)})x^{(0)}\otimes x^{(1)})(1\otimes c_i))w\varphi(bd_i)\\
		&=& \sum_i(yz^{(0)})x^{(0)}\varphi(x^{(1)}c_i)w\varphi(bd_i)\\
		&=& (yz^{(0)})x^{(0)}w(\widehat{\Delta}(\varphi(z^{(1)}\underline{\hspace{0.3cm}} ))(1\otimes\widehat{a}))(x^{(1)}\otimes b)\\
		&=& (yz^{(0)})x^{(0)}w\varphi(z^{(1)}x^{(1)}b_{(1)})\widehat{a}(b_{(2)})\\
		&=& (\imath\otimes\varphi\otimes\widehat{a})(((yz^{(0)}\otimes z^{(1)})\rho(x)(w\otimes 1)\otimes 1)(1\otimes\Delta(b)))\\
		&=& (\imath\otimes\varphi\otimes\widehat{a})(((y\otimes 1)\rho(zx)(w\otimes 1)\otimes 1)(1\otimes\Delta(b)))\\
		&=& (\imath\otimes\varphi\otimes\widehat{a})(((y(zx)^{(0)}\otimes (zx)^{(1)})(w\otimes 1)\otimes 1)(1\otimes\Delta(b)))\\
		&\stackrel{(*)}{=}& (y(zx)^{(0)})w\varphi((zx)^{(1)}b_{(1)})\widehat{a}(b_{(2)}S^{-1}(\delta)\delta)\\
		&=& (y(zx)^{(0)})w\varphi((zx)^{(1)}b_{(1)})\widehat{a}^{\delta}(b_{(2)}S^{-1}(\delta))\\
		&=& (y(zx)^{(0)})w\varphi((zx)^{(1)}b_1)S^{-1}(\widehat{a}^{\delta})(\delta S(b_{(2)}))\\
		&\stackrel{(\ref{grouplike})}{=}& (y(zx)^{(0)})w S^{-1}(\widehat{a}^{\delta})(((\varphi\otimes \imath)\Delta((zx)^{(1)}b_{(1)}))S(b_{(2)}))\\
		&=& (y(zx)^{(0)})w S^{-1}(\widehat{a}^{\delta})(((zx)^{(1)})_{(2)})\varphi(((zx)^{(1)})_{(1)}b)\\
		&=& (\imath\otimes\varphi\otimes S^{-1}(\widehat{a}^{\delta}))((y(zx)^{(0)})w\otimes \Delta((zx)^{(1)})(b\otimes 1))\\
		&=& y(zx)^{(0)}w\varphi(((zx)^{(1)})_{(1)}b)S^{-1}(\widehat{a}^{\delta})(((zx)^{(1)})_{(2)})\\
		&\stackrel{(**)}{=}& y(zx)^{(0)}w\varphi(((zx)^{(1)})_{(1)}b)\varphi(((zx)^{(1)})_{(2)}c)\\
		&=& (\imath\otimes\varphi\otimes\varphi)((\imath\otimes\Delta)((y\otimes 1)\rho(zx))(w\otimes b\otimes c)),
	\end{eqnarray*}
where, in equality $(*)$ we used the distinguished group-like element $\delta \in M(A)$ as defined in (\ref{grouplike}) and in $(**)$ we defined the element $c\in A$ such that 
\begin{equation} \label{elementc}
S^{-1}(\widehat{a}^{\delta})= \varphi (\underline{\hspace{0.3cm}} c).
\end{equation}
 	
On the other hand,
\begin{eqnarray*}
	[y,z\triangleleft(x\#\widehat{a})](w\otimes b)&=&[y,S^{-1}(\widehat{a}^{\delta})\cdot (zx)](w\otimes b)\\
	&\stackrel{(\ref{elementc})}{=}& [y,\varphi(\underline{\hspace{0.3cm}} c)\cdot (zx)](w\otimes b)\\
	&=& (y((zx)^{(0)})^{(0)}\#\varphi(((zx)^{(0)})^{(1)}\underline{\hspace{0.3cm}} ))\varphi((zx)^{(1)}c)(w\otimes b)\\
	&=& y((zx)^{(0)})^{(0)}w\varphi(((zx)^{(0)})^{(1)}b)\varphi((zx)^{(1)}c)\\	
	&=& (\imath\otimes\varphi\otimes\varphi)((y\otimes 1\otimes 1)(\rho\otimes \imath)(\rho(zx)(1\otimes c))(w\otimes b\otimes 1))\\
	&\stackrel{(\ref{escritacoa3})}{=}& (\imath\otimes\varphi\otimes\varphi)((y\otimes 1\otimes 1)(\imath\otimes\Delta)(\rho(zx))(E(w\otimes b)\otimes c))\\
	&=& (\imath\otimes\varphi\otimes\varphi)((\imath\otimes\Delta)((y\otimes 1)\rho(zx))(w\otimes b\otimes c)),				
	\end{eqnarray*}
	where in the last equality we are using $w\otimes b\in E((\widehat{A}\cdot R)\otimes A)$. Then $[y,z](x\#\widehat{a})=[y,z\triangleleft(x\#\widehat{a})]$. Therefore, $[ \ , \ ]$ is bilinear.
\end{proof}

Now, we have all the necessary conditions to construct the corresponding Morita context.

\begin{thm} Let $(R, \rho, E)$ be a restrict symmetric partial $A$-comodule algebra. Then
	$$((\widehat{A}\cdot R)\#\widehat{A},R^{\underline{coA}},{}_{(\widehat{A}\cdot R)\#\widehat{A}}(\widehat{A}\cdot R)_{R^{\underline{coA}}},{}_{R^{\underline{coA}}}(\widehat{A}\cdot R)_{(\widehat{A}\cdot R)\#\widehat{A}},[ \ , \ ],( \ , \ ))$$
	is a Morita context.
	\label{contMoritparc}
\end{thm}
\begin{proof}
	By the results \ref{bimodu1}, \ref{bimodu2}, \ref{apli_parent} and \ref{aplic_colche} it only remains to verify the compatibility conditions 
	\begin{eqnarray}
	[x,y]\triangleright z&=& x\triangleleft(y,z) \label{cont_comp1}\\
	(x,y)\triangleright z&=& x\triangleleft[y,z], \label{cont_comp2}
	\end{eqnarray}
	for all $x, y, z\in\widehat{A}\cdot R$.
	
	Indeed, for  $x,y,z\in\widehat{A}\cdot R$, write $(y\otimes 1)\rho(z)=\sum_iy_i\otimes a_i$ and for each $i$ define the elements $b_i\in A$ such that \begin{equation}\label{bi}
	\varphi(a_i \underline{\hspace{0.3cm}})=\varphi(\underline{\hspace{0.3cm}} b_i). 
	\end{equation}
	Then
	\begin{eqnarray*}
		&&\hspace{-1.8cm} r(x\triangleleft[y,z])=\\
		\hspace{0.7cm}&=& r(x\triangleleft(yz^{(0)}\#\varphi(z^{(1)}\underline{\hspace{0.3cm}})))\\
		&=&\sum_i r(x\triangleleft(y_i\#\varphi(a_i\underline{\hspace{0.3cm}} )))\\		
		&\stackrel{(\ref{bi})}{=}&\sum_i r(x\triangleleft(y_i\#\varphi(\underline{\hspace{0.3cm}} b_i)))\\
		&=&\sum_i r(S^{-1}(\widehat{b_i}^{\delta})\cdot (xy_i))\\
		&=& 	\sum_i r(\varphi(\underline{\hspace{0.3cm}} c_i)\cdot (xy_i))\\
		&=& \sum_i(\imath\otimes \varphi)((r\otimes 1)\rho(xy_i)(1\otimes c_i))\\
		&=& \sum_i r(xy_i)^{(0)}S^{-1}(\widehat{b_i}^{\delta})((xy_i)^{(1)})\\
		&=& \sum_i r(xy_i)^{(0)}\varphi(S^{-1}((xy_i)^{(1)})\delta b_i)\\
		&\stackrel{(\ref{bi})}{=}& \sum_i r(xy_i)^{(0)}\varphi(a_iS^{-1}((xy_i)^{(1)})\delta)\\
		&=& (\imath\otimes\varphi(\underline{\hspace{0.3cm}} \delta))(\imath\otimes m\circ\sigma\circ(S^{-1}\otimes \imath))((r\otimes 1\otimes 1)(\rho(x)\otimes 1)(\rho\otimes \imath)((y\otimes 1)\rho(z)))\\
		&\stackrel{\ref{reduzi}}{=}& (\imath\otimes\varphi(\underline{\hspace{0.3cm}} \delta))(\imath\otimes m\circ\sigma\circ(S^{-1}\otimes \imath))((r\otimes 1\otimes 1)(\rho(x)\rho(y)\otimes 1)(\imath\otimes\Delta)(\rho(z)))\\
		&=& r(xy)^{(0)}\varphi(S^{-1}((xy)^{(1)})\delta)z\\
		&=& r((xy)^{(0)}z)\varphi(S^{-1}((xy)^{(1)})\delta)\\
		&\stackrel{(\ast)}{=}& r(((xy)^{(0)}z)\varphi((xy)^{(1)}))\\
		&=& r((x,y)\triangleright z),
	\end{eqnarray*}
	for all $r\in R$, where the equality $(\ast)$ is ensured by $\varphi(S(a))=\varphi(a\delta)$.
	
	The other condition of compatibility follows trivially. 
\end{proof}

\begin{pro}
	Under the conditions of Theorem \ref{contMoritparc}, if the linear maps $[ \ , \ ]$ and  $( \ , \ )$  are surjective, then they are  injective.
	\label{Pro_sobrinje}
\end{pro}
\begin{proof}
	In order to show the injectivity of the linear map $[ \ , \ ]$, we will use Lemma \ref{bimodu2} to define the right action $\blacktriangleleft$ of $((\widehat{A}\cdot R)\#\widehat{A})$ on $((\widehat{A}\cdot R)\otimes_{R^{\underline{coA}}} (\widehat{A}\cdot R))$ as follows,
	\begin{eqnarray*}
		(x\otimes y)\blacktriangleleft(z\#\widehat{a})=x\otimes(y\triangleleft (z\#\widehat{a}))=x\otimes (S^{-1}(\widehat{a}^{\delta})\cdot (yz)),
	\end{eqnarray*}
	for all $x,y,z\in \widehat{A}\cdot R$ and $\widehat{a}\in\widehat{A}$. Then $\blacktriangleleft$ is a unitary nondegenerate right module.
	
	Now, assume that $\sum_ix_i\otimes y_i\in ker[ \ , \ ]$, i.e., $\sum_i[x_i,y_i]=0$. By the surjectivity of the linear map $[ \ , \ ]$ we can write $\sum_jz_j\#\widehat{a_j}=\sum_k[r_k,s_k]\in (\widehat{A}\cdot R)\#\widehat{A}$, thus
	\begin{eqnarray*}
		(\sum_ix_i\otimes y_i)\blacktriangleleft(\sum_jz_j\#\widehat{a_j})&=& \sum_ix_i\otimes (y_i\triangleleft\sum_jz_j\#\widehat{a_j})\\
		&=& \sum_ix_i\otimes y_i\triangleleft \sum_k[r_k,s_k]\\
		&\stackrel{(\ref{cont_comp2})}{=}& \sum_ix_i\otimes \sum_k(y_i,r_k)\triangleright s_k\\
		&=& \sum_{i,k} x_i\triangleleft(y_i,r_k)\otimes s_k\\
		&\stackrel{(\ref{cont_comp1})}{=}& \sum_{i,k} [x_i,y_i]\triangleright r_k\otimes s_k\\
		&=& 0,
	\end{eqnarray*}
	which means $(\sum_ix_i\otimes y_i)\blacktriangleleft(\sum_jz_j\#\widehat{a_j})=0$, for all $\sum_jz_j\#\widehat{a_j}\in (\widehat{A}\cdot R)\#\widehat{A}$. Therefore,  $\sum_ix_i\otimes y_i=0$.
	
	In a similar way, one can show the injectivity of the linear map $( \ , \ )$.
\end{proof}

\subsection{Galois Coaction}

Our goal is to connect the Morita context, constructed in the previous section, with the Galois theory inherent, as was made in the global case in \cite{Galois}. In what follows, $A$ is assumed to be  a regular multiplier Hopf algebra with integrals and $(R,\rho,E)$ a symmetric partial $A$-comodule algebra.

\begin{defi}\label{def_galois}
	We call $\rho$ a \textbf{partial Galois coaction} if $\rho$ is restrict and the linear map
	\begin{eqnarray*}
		\beta:(\widehat{A}\cdot R)\otimes_{R^{\underline{coA}}}(\widehat{A}\cdot R) & \longrightarrow & ((\widehat{A}\cdot R)\otimes A)E\\
		x\otimes y & \longmapsto & (x\otimes 1)\rho(y)
	\end{eqnarray*}
	is bijective.
\end{defi}

\begin{thm} If $(R,\rho,E)$ is a restrict symmetric partial $A$-comodule algebra, then, the following conditions are equivalents:
	\begin{enumerate}
		\item[(i)] $\rho$ is a partial Galois coaction;
		\item[(ii)] $\beta$ is surjective;
		\item[(iii)] $[ \ , \ ]$ is surjective.
	\end{enumerate}
\end{thm}		
\begin{proof} 
	
	(i)$\Rightarrow$ (ii) Follows from Definition \ref{def_galois}.
	
	(ii)$\Rightarrow$ (iii) Consider the linear map
	\begin{eqnarray*}
		\alpha : ((\widehat{A}\cdot R)\otimes A)E &\longrightarrow & ((\widehat{A}\cdot R)\#\widehat{A})|_{E((\widehat{A}\cdot R)\otimes A)}\\
		(x\otimes a)E &\longmapsto & x\#\varphi(a \underline{\hspace{0.3cm}} ).
	\end{eqnarray*}
	It is straightforward to check that  $\alpha$ is bijective.
	
	Notice that, since $[ \ , \ ] = \alpha\circ\beta$, then $[ \ , \ ]$ is surjective.
	
	(iii)$\Rightarrow$ (i) Suppose that the linear map $[ \ , \ ]$ is surjective hence, by Proposition \ref{Pro_sobrinje}, $[ \ , \ ]$ is bijective. Therefore, $\beta=\alpha^{-1}\circ [ \ , \ ]$ is bijective.
\end{proof}

\section{Acknowledgments}
The authors would like to thank to Alfons Van Daele for his kind solicitude in explaining the theory, leading to a better understanding of multiplier Hopf algebras. Also to Antonio Paques for fruitful discussions on the theory. 


\end{document}